\documentclass[11pt,oneside,a4paper]{article}
\usepackage[T2A]{fontenc}
\usepackage[utf8]{inputenc}
\usepackage[russian, english]{babel}
\usepackage{color, epsfig, amsmath, amssymb, bm, graphicx, float, mathtools, amsthm, textcomp, stackengine, enumitem, comment, tcolorbox, lipsum}

\graphicspath{ {images/} }

\usepackage{geometry}
\tcbuselibrary{theorems}
\usepackage[framemethod=TikZ]{mdframed}

\newtheorem{theorem}{Theorem}[section]

\newtheorem{lemma}{Lemma}[section]
\newtheorem{proposition}{Proposition}[section]
\newtheorem{definition}{Definition}[section]

\usepackage{tikz}
\usetikzlibrary{decorations.pathreplacing}
\usetikzlibrary{arrows, positioning}
\tikzset{
    >=stealth',
    punkt/.style={
           rectangle,
           rounded corners,
           draw=black, very thick,
           text width=6.5em,
           minimum height=2em,
           text centered},
    pil/.style={
           ->,
           thick,
           shorten <=2pt,
           shorten >=2pt,}
}

\stackMath

\DeclareMathOperator*\uplim{\overline{lim}}
\DeclareMathOperator*\lowerlim{\underline{lim}}

\newtheoremstyle{break}
{\topsep}{\topsep}%
{\itshape}{}%
{\bfseries}{}%
{\newline}{}%
\theoremstyle{break}

\usepackage{hyperref}
\hypersetup{
    colorlinks=true,
    citecolor=black,
    linkcolor=black,
    filecolor=magenta,      
    urlcolor=cyan,
    pdftitle={LDP},
    pdfpagemode=FullScreen,
    }
    
\usepackage{ragged2e}
\usepackage{manyfoot}
\SetFootnoteHook{\hspace*{-1.8em}}
\DeclareNewFootnote{bl}[gobble]
\title{A large deviation principle for the Schramm-Loewner evolution in the uniform topology}
\date{\empty}
\author{Vladislav Guskov}

\begin{document}

\clearpage
\maketitle
\thispagestyle{empty}
\begin{center}\justifying
\textbf{Abstract}. We establish a large deviation principle for chordal SLE$_\kappa$ parametrized by capacity, as the parameter $\kappa \to 0+$, in the topology generated by uniform convergence on compact intervals of the positive real line. The rate function is shown to equal the Loewner energy of the curve. This strengthens the recent result of E. Peltola and Y. Wang who obtained the analogous statement using the Hausdorff topology. 
\end{center}
\footnotebl{Department of mathematics, KTH Royal Institute of Technology, 100 44,  Stockholm, Sweden.\\ E-mail: \texttt{vguskov@kth.se}}
	\section{Introduction}	The Schramm-Loewner evolution with parameter $\kappa > 0$, subsequently referred to as SLE$_\kappa$, is a fractal random curve that connects two marked points $a,b \in \partial D$ on the boundary of a simply connected domain $D$ in the complex plane. SLE curves arise as scaling limits of interfaces in 2D critical lattice models, see, e.g., \cite{LSW04, Smirnov10, smirnov2001critical, Schramm_Sheffield_2005}, and play an important part in the analysis of the geometry of the Gaussian free field \cite{Schramm_Sheffielde_2009, miller2016imaginary}. 
	
	In the last few years, starting with Y. Wang's paper \cite{Wang2019}, there has been substantial interest in various questions related to large deviations of SLE$_\kappa$ as the parameter $\kappa$ tends to $0+$ as well as $+\infty$. One reason for this is that the large deviations rate functions that appear in such statements turn out to be interesting quantities that somewhat surprisingly provide links between random conformal geometry and other areas such as Teichm\"uller theory and related parts of analysis and geometry, see \cite{Wang2022survey}. The key quantity in this paper is the \emph{Loewner energy}, defined below for chords, which enters the story as the large deviations rate function for SLE$_\kappa$ as $\kappa \downarrow 0$. (Given a domain $D$ with given boundary points $a,b$, a \emph{chord} is a simple curve in $D$ connecting $a$ and $b$, otherwise staying in $D$.)

	Chordal SLE curves are constructed using the Loewner differential equation 
	\[
	\partial_t g_t(z) = \frac{2}{g_t(z) - \lambda(t)}, \quad g_0(z) = z \in \mathbb{H}.
	\] 
	Here, $(\lambda(t))_{t \ge 0}$ is a continuous real-valued (driving) function and, for each $t\ge 0$, the solution $g_t$ is a conformal map from a simply connected subset of the upper half-plane $\mathbb{H}$ onto $\mathbb{H}$. If $\lambda$ is smooth enough, $\gamma(t) = \lim_{y \to 0+} g_t^{-1}(\lambda(t) + iy), \, t\ge 0,$ defines a chord in $\mathbb{H}$ connecting the origin with infinity.  The parametrization for the curve that one gets from this construction is called the \emph{capacity parametrization}. Conversely, starting with a chord $\gamma$ connecting $0$ with $\infty$ in $\mathbb{H}$, one can recover the Loewner driving function via the equation on the maps $g_t: \mathbb{H} \setminus \gamma[0,t] \to \mathbb{H}$, where the curve is assumed to be parametrized so that $\textrm{hcap}(\gamma[0,t]) =2t$, i.e., near infinity the map has the following expansion $g_t(z) = z + 2t/z + o(1/|z|)$ as $z \to \infty$. In the case of SLE$_\kappa$, one takes the driving function to be rescaled standard Brownian motion $(\sqrt{\kappa}B(t))_{t \ge 0}$. It is a non-trivial fact that the above limit defines a curve in this case, see \cite{RohdeSchramm}. 
	
	The analytic behavior of SLE$_\kappa$ curves depends strongly on the $\kappa$-parameter. For example, with probability one, the Hausdorff dimension of the curve equals $\min \left(1+\kappa/8, 2\right)$, see \cite{Beffara}, and it is known that for $\kappa \le 4$ the SLE curve is almost surely simple \cite{RohdeSchramm}; in what follows we consider only this regime. 
	
	Now we give a heuristic description of large deviations for SLE. In the limit as $\kappa \downarrow 0$, it is not hard to see that the SLE$_\kappa$ curve converges to the deterministic hyperbolic geodesic chord $\eta$ connecting the two marked boundary points in $D$; by definition this is the image of the imaginary axis under a conformal map from the upper half-plane onto $D$ taking $0$ to $a$ and $\infty$ to $b$. Hence if we take a suitable family $V$ of chords that does not include $\eta$, then			\begin{equation*}
					\mathbb{P}\left[\text{SLE}_\kappa \in V\right] \to 0 \text{ as }\kappa\downarrow 0.
			\end{equation*}

The leading order convergence rate is provided by a \textit{large deviation principle}, abbreviated LDP in the sequel. Roughly speaking, the probability decays exponentially fast in $1/\kappa$:
				\begin{equation*}
					\mathbb{P}\left[\text{SLE}_{\kappa} \in V\right] \approx e^{-\frac{I(V)}{\kappa}} \text{ as } \kappa \downarrow 0.
				\end{equation*} 
			The rate of decay is given by 
			\begin{equation*}
			I\left(V\right) = \inf\limits_{\gamma \in V}I_{L}\left(\gamma\right),
	 		\end{equation*}
			where the \emph{rate function} \[I_L(\gamma) = \frac{1}{2} \int_0^\infty {\dot\lambda}(t)^2 dt \] is the Loewner energy introduced by Y. Wang in \cite{Wang2019} and subsequently studied in a number of papers, see, e.g., \cite{Wang_equivalent, Viklund_Wang_Interplay, Viklund_Wang_Loewner-Kufarev, Rohde_Wang_Loewner_energy}. Here, $\lambda$ is the Loewner driving function for $\gamma$, see below for further comments on this definition.
				
	In the present paper $I_L$ only enters as the rate function and we will not consider its other interpretations. Let us however very briefly mention some of the known facts. First, we note that chords with finite Loewner energy are quite smooth: they are known to be rectifiable quasi-slits, i.e., quasiconformal images of the geodesic $\eta$, but not necessarily $C^1$, see \cite{Rohde_Wang_Loewner_energy, PeltolaWang}. Remarkably, the Loewner energy is invariant with respect to reversing the curve, see \cite{Wang2019}. It is possible to define the Loewner energy for Jordan curves \cite{Rohde_Wang_Loewner_energy}, which in this case is M\"obius invariant. In this more general setting the family of finite energy curves can be identified with the class of Weil-Petersson quasicircles. The Loewner energy turns out to be (a constant times) the K\"ahler potential for the Weil-Petersson metric on the Weil-Petersson Teichm\"uller space, i.e., the set of Weil-Petersson quasicircles viewed as elements of the universal Teichm\"uller space, see \cite{Wang_equivalent, Bishop, Takhtajan_Teo}. See also \cite{Johansson} for the emergence of the Loewner energy in the context of the Szeg{\H o} theorem.

			\subsection{Main result}
			We first give the precise meaning of large deviation principle for SLE assuming a suitable topological space $(S,\tau)$ has been chosen. 
	\begin{definition}\label{definition_of_LDP}
	We say that the chordal SLE$_\kappa$ curve $\gamma^\kappa$ in a given topological space $\left(S, \tau\right)$ satisfies a large deviation principle with rate function $I_L:S\to[0,\infty]$ if  
		\begin{center}
	  		1. $\uplim\limits_{\kappa \downarrow 0}\kappa\log\mathbb{P}\left[\gamma^\kappa \in F\right]\le - I_L(F)$ for any closed subset $F$ of $S$,\\
			2. $\lowerlim\limits_{\kappa\downarrow 0}\kappa\log\mathbb{P}\left[\gamma^\kappa \in G\right]\ge - I_L(G)$ for any open subset $G$ of $S$,
		\end{center}
		where $I_L(V) = \inf\limits_{\gamma\in V}I_L(\gamma)$ for any subset $V\subset S$. 
	\end{definition}
	   There are several natural choices of the topological space in the case at hand. From Schilder's theorem (see below) and the contraction principle (Theorem \ref{contraction principle} in Section \ref{section_proof}) one immediately obtains an LDP using the Carath\'eodory topology on the conformal maps solving the Loewner equation.  E. Peltola and Y. Wang obtained an LDP for SLE$_\kappa$ curves viewed as sets, using the topology induced by the Hausdorff metric on compact subsets. (See also \cite{Wang2019} for another version based on the prescribed left/right passage given marked points in $\mathbb{H}$.) 
	   
	   In the present paper we establish an LDP for SLE$_\kappa$ curves, as $\kappa \downarrow 0$, viewed as continuous curves in the capacity parametrization, using the topology generated by uniform convergence on compact intervals which in this context we shall call the \textit{uniform topology}. This strengthens the result obtained in \cite{PeltolaWang} and places the LDP for SLE in perhaps a more natural setting. Moreover, in view of applications it is technically useful to have this stronger LDP at hand: for example, there are important ``observables'' that are continuous in the uniform topology but not in the Hausdorff topology. One instance is the harmonic measure of the left side of the curve, see \cite{Krusell} for detailed description. We will discuss other possible choices for the topology in Section~\ref{sect:further-comments}.

	 In order to state our main result, let $S$ denote the space of  continuous capacity-parameterized simple curves in the upper half-plane started at the origin
	 \begin{equation}\label{definition_S}
	 \begin{split}
	S = \big\{\gamma \in C\left([0,\infty), \overline{\mathbb{H}}\right):& \gamma(0)=0, \gamma[0,\infty) \text{ is a simple curve},\\ &\text{ and hcap}(\gamma[0,t]) = 2t \text{ for } t\in [0,\infty) \big\}.
	 \end{split}
	 \end{equation}
	 We endow this space with the topology $\tau$ of uniform convergence on compact intervals of the positive real line (compact convergence). 	 
	 \begin{theorem}\label{LDP_theorem}
		As $\kappa \downarrow 0$, SLE$_\kappa$ satisfies the large deviation principle in the topological space $\left(S, \tau\right)$ with the Loewner energy $I_L$ as a good rate function\footnote{The notion of \textit{good rate function} is explained in Section \ref{section_rate_function}.}.
	 \end{theorem}
	 Although the space $S$ is not complete, it is a natural choice for our setup since, for $\kappa\le4$, SLE$_\kappa$ curves are simple almost surely. 
	 
	 As a side remark, we would like to mention another formulation of the result.  It would have been quite natural to state the theorem using the (slightly weaker) ``strong topology'' on curves modulo reparametrization: We say that the sequence of curves $\left\{\gamma_n\right\}_n$ converges to $\gamma$ in the strong sense if 
\begin{equation*}
\lim\limits_{n\to\infty}d_T\left(\gamma_n, \gamma\right)=0,
\end{equation*}
where the metric is given by 
	 \begin{equation*}
	d_T(\gamma_1, \gamma_2) = \inf\limits_{\varphi}||\gamma_1 - \gamma_2 \circ\varphi||_{\infty, [0,T]},  
	 \end{equation*}
	  the infinum is taken over all increasing homeomorphisms $\varphi: [0,T]\to[0,T]$. Analogously to compact convergence, the strong convergence for chords with infinite capacity is defined by saying that the strong convergence takes places for the restrictions of curves to the initial segments with
	 half-plane capacity $2T$ for every $T>0$. We choose however to use the capacity parametrization in the statement since this is what is actually proved.
	 
	 \section*{Acknowledgments}
	 I would like to thank Fredrik Viklund for helpful discussions and numerous revisions of this work as well as constant motivation in the process. I am grateful to Ellen Krusell for many valuable comments and inspiring conversations. Furthermore, I also want to thank Yilin Wang and Yizheng Yuan who read the draft of this work and provided insightful remarks. The work is supported by the Swedish Research Council (VR), Grant 2019-04152.

\newpage
	\section{Setup and outline of the proof}
	
	M. Schilder showed in 1966 that a family of rescaled Brownian motions $\left\{\sqrt{\varepsilon}B\right\}_{\varepsilon>0}$ satisfies an LDP on the Polish space\footnote{\label{note1} Complete separable metric space.} $\left(C([0,T], \mathbb{R}), ||\cdot||_{\infty}\right)$ with the Dirichlet energy $I_D$ as a rate function. The Dirichlet energy is directly connected with the Loewner energy $I_L$ of a chord $\gamma$ in the upper half-plan driven by $\lambda$. By definition:
	\begin{equation*}
	I_L(\gamma)  = I_D(\lambda).
	\end{equation*} This connection
	gives hope to deduce LDP for SLE from LDP for Brownian motion by means of the contraction principle. In order to carry out this strategy the Loewner map, from a space of driving functions  to the space of curves, has to be continuous. The continuity of the  map depends on the chosen topology on the space of curves. Unfortunately the Loewner map is continuous neither in the Hausdorff topology nor in the uniform topology which makes direct application of the contraction principle impossible. Therefore, the proof requires direct validation of the inequalities for open and closed sets as stated in Definition \ref{definition_of_LDP} above.

The proof of the main result is divided into two parts: the proof for a finite time interval $[0,T]$ and its extension to $[0,\infty)$. The former utilizes analytic results specific to SLE while the latter follows from the general Large Deviations theory for projective limits.

Since the proof for finite time will constitute the bulk of the paper, 
 for the most part we will work with the space of curves run up to finite capacity time:
	
	\begin{equation}\label{definition_S_T}
	\begin{split}
	S_T = \big\{\gamma \in C\left([0,T], \overline{\mathbb{H}}\right):&   \gamma(0)=0, \gamma[0,T] \text{ is a simple curve}, \\
	&\text{ and hcap}(\gamma[0,t]) = 2t \text{ for } t\in [0,T] \big\}.	
	\end{split}
	\end{equation}
Moreover,  it is sufficient to consider the case $T=1$ due to the scale invariance of SLE$_{\kappa}$.

Let $D_T $ be the biggest subset of $ C\left([0,T], \mathbb{R}\right)$ such that every $\lambda \in D_T$ generates a curve under the Loewner transformation $\lambda\to \mathcal{L}(\lambda)$. Throughout the paper we will use the following notation for SLE$_{\kappa}$: $\gamma^\kappa = \mathcal{L}\left(\sqrt{\kappa }B\right)$, where $B$ is the standard Brownian motion, or just $\gamma$ if it is clear that we work with SLE$_\kappa$ from the context.

 The Loewner differential equation for the inverse conformal map $f_t = g_t^{-1}$, which maps $\mathbb{H}$ to a slit domain $ \mathbb{H}\backslash\gamma[0,t] $, is given by

\begin{equation}\label{Loewner_equation_for_f}
\partial_{t}f_t(z) = - f'_t(z)\frac{2}{z - \lambda(t)}, \quad f_0(z) = z.
\end{equation}
Given a real-valued function with appropriate regularity one obtains a growing chord $\left\{\gamma(t)\right\}_{t\in[0,T]}$ via $\gamma(t) = \lim_{y \to 0+} f_t(\lambda_t + iy)$. Above we have denoted by $D_T$ the class of driving functions that generate curves; we stress however that there is currently no known direct characterization of this class and finding one is an interesting open problem. Nevertheless, some sufficient conditions are available. For example, it is known that if a continuous function $\lambda:[0,T]\to\mathbb{R}$  has finite Dirichlet energy, then it generates a simple curve \cite{FrizShekhar}. This case will be of most interest to us. 

Under $f_t$ the point $\lambda_t$  on the real line is mapped to the tip of the curve $\gamma(t)$. If we introduce the centered map $\hat{f}_t(z) := f_t(\lambda_t+z)$, then zero will be the preimage of the tip $\gamma(t)$ under $\hat{f}_t$.  Moreover, in further derivations $f_t$, without any additional indices, will denote the conformal map for SLE$_\kappa$, that is the solution to (\ref{Loewner_equation_for_f}) with $\lambda= \sqrt{\kappa}B$.

The main part of the proof follows the logic of standard arguments to derive sample path large deviations. First, we show that SLE satisfies LDP on $S_T$ in the topology generated by the supremum norm. The general line of reasoning follows ideas used in the proof of Schilder's theorem as presented, for example, in the monograph of  A.V. Bulinski and \mbox{A. N. Shiryaev} \cite{BulinskiShiryaev}. The cornerstone of the SLE specific argument is that, for any $\beta \in (0,1)$ and certain constant $c>0$, the derivative of the SLE map satisfies the inequity $\left|\hat{f}'_t(iy)\right| \le cy^{-\beta}$ uniformly in $t\in [0,T]$ with very high probability if $\kappa$ is small enough.  The reason why this is useful is that the derivative near the preimage of the tip of the curve controls the regularity in the capacity parametrization, as well as various related quantities, see, e.g., \cite{Optimal, ViklundLawler_Acta}.

The proof for closed sets is based on a classic idea of approximating stochastic processes. For example, in the proof of Schilder's theorem Brownian motion $B$ can be approximated by a piece-wise linear function $B_n$ with $n$ nodes. The limit $\lim\limits_{n\to\infty}B_n = B$ a.s.  allows us to work with approximating processes $B_n$ instead of $B$ which makes certain calculations possible that would produce infinite quantities otherwise. Similarly, we are going to approximate the SLE$_\kappa$ curve $\gamma^\kappa$ with a certain processes $\gamma_{n}$. Of course, there could be many candidates for these processes. As was shown in \cite{Tran} if $\gamma_n$ is a Loewner curve driven by appropriate approximation of $\sqrt{\kappa}B$, then almost surely $\lim\limits_{n\to\infty}||\gamma^{\kappa} - \gamma_n||_{\infty} = 0$. Among many possible choices of approximating processes, piece-wise linear approximation of the driving function proved to be the most fruitful. 

An important technical building block in the proof of convergence of $\gamma_n$ to SLE$_\kappa$ is the derivative estimate mentioned above. More specifically, one needs to control the lower bound of the probability of the event
\begin{equation}\label{definition_P}
P_{n} \overset{\mathrm{def}}{=}  \left\{|\hat{f}'_{t}(iy)|\le \psi(n)y^{-\beta} \text{ for } y\in[0, 2^{-n}], t\in[0,1] \right\},
\end{equation} 
for a certain sub-power function (see below) $\psi$ and a free parameter $\beta \in (0,1)$. It will allow us to control the convergence rate of  $||\gamma^{\kappa} - \gamma_n||_{\infty}$, where $\gamma_n  = \mathcal{L}\left(\sqrt{\kappa}B_n\right)$ and $B_n$ is  a piece-wise linear approximation of Brownian motion with $n$ nodes. 

The proof for open sets partly follows the derivation of Schilder's theorem. It is possible to use more or less the same argument due to a result obtained in \cite{Continuity} which allows precise comparing of Loewner maps with close driving functions.  

The extension of the LDP to the positive real line $[0,\infty)$ can be done with tools developed in the  general theory of Large Deviations (see a book of A. Dembo and O. Zeitouni \cite{DemboZeitouni} on the subject). Namely, the Dawson-G$\ddot{\text{a}}$rtner theorem will allow us to extent the result from $S_T$ to the projective limit of the family $\left\{S_T\right\}_{T\in[0,\infty)}$. The projective limit is a product space with the product topology. Here is the place where the topology of uniform convergence on $[0,T]$ becomes a version of the topology of compact convergence on $[0,\infty)$. The product space is just a middle step since, at the end of the day, we are interested in the LDP on $S$. The final step is carried out with a help of the contraction principle; we will see that there is a continuous bijection between the product space and $S$ which makes the use of the contraction principle possible.

	Here is a brief description of how this paper is organized.  In Section \ref{section_rate_function} we show that the Loewner energy satisfies general definition of a rate function from Large Deviations theory; the section ends with one useful property of a rate function.  Next, in Sections \ref{section_derivative_estimate} and \ref{section_convergence_to_SLE} we establish derivative estimates on which the main argument in Section \ref{section_proof} will be build.  Section \ref{section_proof} consists of three parts: proof for finite time, which is subdivided into cases of closed and open sets, and extension to $[0,\infty)$.
	
\section{The Loewner energy as rate function}\label{section_rate_function}
Definition \ref{definition_of_LDP} of LDP features a rate function. Here we give its formal definition, prove that the Loewner energy satisfies this definition and verify one limiting property which will be used later.  
\begin{definition}\label{definition_of_rate_function}
A function $I: S\to [0,\infty]$ is said to be a good rate function if it satisfies the following properties
\begin{enumerate}
\item $I \not\equiv \infty$
\item $I$ is lower semi-continuous 
\item $\left\{\lambda: I(\lambda)\le c \right\}$ is compact for every $c\in[0,\infty)$.
\end{enumerate}
\end{definition}
\noindent In fact, the second property follows from the third one but is stated together due to convention.

One of the motivations for this definition is based on the uniqueness property of a rate function, i.e., if a family of probability measures satisfies LDP, then the associated rate function is unique.  To obtain this property it is enough to require the underlying space $S$ to be regular and the rate function to be lower semi-continuous. Moreover, requiring compactness of the level sets of the rate function gives additional nice properties (hence, \textit{good} in the name).

 Recall that the \textit{Dirichlet energy} of a function $\lambda \in C(\left[0,T\right], \mathbb{R})$  is given by
\begin{equation*}
I_{D}\left(\lambda\right)= \begin{cases}
\frac{1}{2}\int\limits_{0}^{T}\dot{\lambda}(t)^{2}dt, \text{ if } \lambda\text{ is absolutely continuous}\\
\infty, \text{ otherwise}.
\end{cases}
\end{equation*}
In the introduction we have mentioned that the Dirichlet energy serves as the rate function in the LDP for Brownian motion (Schilder's theorem). In fact, it is a \textit{good} rate function in the sense of Definition $\ref{definition_of_rate_function}$.

\begin{lemma}[Lemma 2, §4, appendix 8 of \cite{BulinskiShiryaev}]\label{Lemma_Ditichlet_energy}
The Dirichlet energy is a good rate function on $C(\left[0,T\right], \mathbb{R})$ endowed with the uniform topology.
\end{lemma}
\noindent The connection between Loewner and Dirichlet energies is given by the following definition.
\begin{definition}
Let $\gamma$ be a Loewner chord in $\overline{\mathbb{H}}$ and $\lambda^{\gamma}$ its driving function. The Loewner energy of $\gamma$ equals the Dirichlet energy of its driving function
\begin{equation*}
I_{L}\left(\gamma\right) = I_{D}\left(\lambda^{\gamma}\right).
\end{equation*} 
\end{definition}

Naturally, one expects the Loewner energy to be a \textit{good} rate function as well. Before proceeding with a proof of this fact we first need to state a couple of preliminary lemmas which will be used repeatedly in this paper. 

\begin{lemma}[Lemma 2.3 in \cite{Continuity}]\label{lemma_continuity}
For $j=1,2$, let $f_t^{(j)}$ satisfy the chordal Loewner equation (\ref{Loewner_equation_for_f}) with a driving function $\lambda_j\in C\left(\left[0,T\right]\right)$. Then for $x+iy \in \mathbb{H}$ the following deterministic inequality holds uniformly in $t\in \left[0,T\right]$
\begin{equation*}
\left|f_t^{(1)}(x+iy) -f_t^{(2)}(x+iy) \right| \le ||\lambda_1 - \lambda_2||_{\infty} \sqrt{1 + \frac{4}{y^2}}.
\end{equation*} 
\end{lemma}

\begin{lemma}[Theorem 2 of \cite{FrizShekhar}]\label{FrizShekhar_lemma}
Let $\lambda\in C\left([0,T]\right)$ be a real-valued function with finite Dirichlet energy and $\hat{f}$ be the corresponding centered Loewner map. Then uniformly in $t\in[0,T]$ and $y>0$ 
\begin{equation*}
\log|\hat{f}'_ t(iy)| \le \frac{1}{2}I_{D}(\lambda).
\end{equation*} 
\end{lemma}

Now we are in shape to prove the following result on the Loewner energy.
\begin{lemma}
The Loewner energy is a good rate function on $\left(S_T, ||\cdot||_\infty\right)$.
\end{lemma}
\begin{proof}

We need to show that for all $c>0$ sets $V_{c} = \left\{\gamma\in S_{T}: I_{L}(\gamma) \le c\right\}$ are compact in $S_{T}$.  The pair $\left(V_c, ||\cdot||_{\infty}\right)$ is a metric space, so compactness is equivalent to sequential compactness.  This being said we proceed by choosing arbitrary sequence of curves  $\left\{\gamma_n\right\}_n$ in $V_c$ and let $\left\{\lambda_n\right\}_n$ be the corresponding driving functions.

There is one-to-one correspondence between curves and driving functions with bounded energies due to the fact that finite energy driving functions generate simple curves (see \cite{FrizShekhar} Theorem 2 or \cite{Wang2019} Proposition 2.1). Hence, the preimage of $V_c$ consists of exactly those driving functions whose Dirichlet energy is less or equal than $c$, i.e.,
\begin{equation*}
\mathcal{L}^{-1}\left(V_{c}\right) = \left\{\lambda \in C\left([0,T], \mathbb{R}\right): I_D(\lambda)\le c \right\}.
\end{equation*}  
By Lemma \ref{Lemma_Ditichlet_energy}, the Dirichlet energy $I_D$ is a good rate function on $\left(C\left([0,T]\right), ||\cdot||_{\infty}\right)$ which, by the definition of a good rate function, implies that the preimage $\mathcal{L}^{-1}\left(V_{c}\right)$ is compact in $C\left([0,T]\right)$. Hence, there exists a subsequence $\left\{n_j\right\}_j$ such that $\left\{\lambda_{n_j}\right\}_j$ converges uniformly to some limiting function $\lambda\in\mathcal{L}^{-1}\left(V_{c}\right)$.

Now we show that corresponding subsequence of curves $\left\{\gamma_{n_j}\right\}_j$ converges uniformly to $\gamma=  \mathcal{L}\left(\lambda\right) \in V_c$. Fix $y>0$ and consider a decomposition
\begin{equation*}
\left|\gamma(t) - \gamma_{n_j}(t)\right| \le \left|\gamma(t)  - \hat{f}_t(iy)\right| +  \left|\gamma_{n_j}(t)  - \hat{f}^{({n_j})}_t(iy)\right| + \left|\hat{f}_t(iy)- \hat{f}^{({n_j})}_t(iy)\right|.
\end{equation*}
The first two terms are bounded with a help of Lemma \ref{FrizShekhar_lemma} 
\begin{equation*}
\left|\gamma(t)  - \hat{f}_t(iy)\right| +  \left|\gamma_{n_j}(t)  -  \hat{f}^{({n_j})}_t(iy)\right|  \le \int\limits_{0}^{y} \left(\left|\hat{f}'_t(ir)\right|dr + \left|(\hat{f}^{({n_j})}_t)'(ir)\right|\right)dr \le 2y e^{c/2}.
\end{equation*}
The third term is bounded by Lemma \ref{lemma_continuity} 
\begin{equation*}
\left|\hat{f}_t(iy)- \hat{f}^{({n_j})}_t(iy)\right| \le ||\lambda - \lambda_{n_j}||_{\infty}\sqrt{1 + \frac{4}{y^2}}.
\end{equation*}
Hence, for a fixed $y>0$ and uniformly in $t$ we have
\begin{equation*}
\lim\limits_{j\to\infty} \left|\hat{f}_t(iy)- \hat{f}^{({n_j})}_t(iy)\right|  = 0.
\end{equation*}
Put together it yields a bound 
\begin{equation*}
\lim\limits_{j\to\infty}||\gamma-\gamma_{n_j}||_{\infty} \le 2ye^{c/2}.
\end{equation*}
Since the choice of $y>0$ was arbitrary we conclude 
\begin{equation*}
\lim\limits_{j\to\infty}||\gamma-\gamma_{n_j}||_{\infty}  = 0.
\end{equation*}
To sum up, any sequence of curves $\left\{\gamma_n\right\}_n$ in $\left(V_c, ||\cdot||_{\infty}\right)$ contains a subsequence that converges to a limit in $V_c$. Hence, the set $V_c$ is compact in $S_T$.
\end{proof}

Once again, the Loewner energy $I_{L}\left(\gamma\right)$ of a curve by definition equals the Dirichlet energy $I_D\left(\lambda^{\gamma}\right)$ of its driving function, so from now on we drop  subscripts and distinguish both energies by their arguments. In what follows the convention $I(V) = \inf\limits_{x\in V}I(x)$ is adopted.

In the beginning of this section we mentioned that requiring compactness of the sub-level sets of a rate function provides additional nice properties. One of them is the following lemma which will be used in the proof of LDP for closed sets.
\begin{lemma}[Lemma 1, §3, Appendix 8 of \cite{BulinskiShiryaev}]\label{closed_sets_rate_function}
If I is a good rate function and $F$ is a closed subset of $S_{T}$, then the following limit holds
\begin{equation}
\lim\limits_{\delta\downarrow 0}I\left(F^{\delta}\right)  = I\left(F\right),
\end{equation}
where $F^{\delta}  = \left\{\gamma\in S_T: \exists \tilde{\gamma}\in F, ||\gamma-\tilde{\gamma}||_{\infty}\le \delta\right\}$ is the $\delta$-neighborhood of the set $F$.
\end{lemma}
\begin{proof}
Recall that by our convention the energy of a set is given by the infinum over all curves inside the set
\begin{equation*}
I(F^\delta) = \inf\limits_{\gamma\in F^\delta}I(\gamma).
\end{equation*}
   First assume that $I(F)<\infty$. The family  $\left\{F^\delta\right\}_{\delta>0}$ is nested, so the sequence of real numbers $\left\{I(F^\delta)\right\}_{\delta>0}$ is increasing (as $\delta \downarrow 0$) and bounded above by $I(F)$. By monotone convergence there exits a bounded limit $$\lim\limits_{\delta\downarrow 0} I\left(F^\delta\right)\le I(F).$$

The main objective is to show the opposite inequality. The lower bound is guaranteed by compactness of sub-level sets and can be obtained in the following way. Pick a sequence $\left\{\gamma_n\right\}_{n\in \mathbb{N}}$ that has uniformly bounded energies. It can be done due to an observation 

\begin{equation*}
	\forall n\in\mathbb{N} \ \exists  \gamma_n \in F^{1/n} : I(\gamma_n) \le I\left(F^{1/n}\right) + \frac{1}{n } \le I(F) +  1.
\end{equation*}
The assumption that $I$ is a good rate function implies that its sub-level sets are compact, i.e., for all $c>0$ the sets $\left\{\gamma: I(\gamma)\le c\right\}$  are compact. Hence, the sequence  $\left\{\gamma_n\right\}_{n\in\mathbb{N}}$ we have picked belongs to some compact in $S_T$ and converges along a subsequence $\{n_{j}\}_{j}$ to some $\gamma\in S_T$. In fact, $\gamma \in F$. Since every $\gamma_{n_j}\in F^{1/n_j}$, where $\left\{F^{1/n_j}\right\}_{j\in\mathbb{N}}$ is a family of nested closed sets, the limiting curve belongs to their intersection which coincides with $F$ since the latter is closed
\begin{equation*}
\gamma \in \bigcap\limits_{j=1}^{\infty} F^{1/n_j} = F.
\end{equation*}
The lower semi-continuity of $I$ implies
\begin{equation*}
\lowerlim\limits_{j\to\infty}I(\gamma_{n_{j}})	\ge I(\gamma)
\end{equation*}
which in turn yields the following chain of inequalities
\begin{equation*}
I(F)\le I(\gamma) \le \lowerlim\limits_{j\to\infty}I(\gamma_{n_{j}})	 \le \lowerlim\limits_{j\to\infty}\left( I\left(F^{1/n_j}\right) + \frac{1}{n_j}\right)= \lowerlim\limits_{\delta\downarrow 0} I\left(F^{\delta}\right).		
\end{equation*}
Note that the last equality holds due to existence of the limit, so it does not matter along what subsequence we choose to take that limit. 
Therefore
\begin{equation*}
	I(F)\le \lowerlim\limits_{\delta\downarrow 0}I\left(F^{\delta}\right)
\le\uplim\limits_{\delta\downarrow 0} I\left(F^{\delta}\right)\le I(F).
\end{equation*}
Now let $I(F)=\infty$. Assume that $\left\{I\left(F^{\delta}\right)\right\}_{\delta>0}$  does not tend to $\infty$ as $\delta \downarrow 0$. Following the same logic as above we arrive at contribution.

\end{proof}

\section{Derivative estimate}\label{section_derivative_estimate}
This section shows that for any $\beta\in (0,1)$ the derivative of the inverse Loewner map $\hat{f}_{t}(z) = f_{t}(\sqrt{\kappa}B_{t} + z)$ satisfies a bound of the form $|\hat{f}'_{t}\left(iy\right)| \le c y^{-\beta}$ uniformly  in $t\in\left[0,1\right]$ with very high probability if $\kappa$ is sufficiently small, and where $c$ does not depend on $\kappa$. Throughout we need to be careful to keep track of the $\kappa$-dependence of constants.

\subsection{Moment estimate}
To begin with, we would like to show that for some $p>0$, that depends on $\kappa$, the expectation $\mathbb{E}\left[|\hat{f}'_{t}(iy)|^p\right]$ is bounded uniformly in $t\in\left[0,1\right]$ and $y>0$ .
For this purpose the reverse Loewner flow will be especially convenient.  Let $h_{t}(z)$ be the reverse SLE flow, i.e., the solution to 
\begin{equation*}
\dot{h}_{t}(z) = \frac{-2}{h_t(z)-\sqrt{\kappa}B_t},\quad h_{0}(z) = z.
\end{equation*}
Then for any fixed $t>0$, the function $z\mapsto h _{t}(z) $ has the same distribution as that of $z\mapsto \hat{f}_{t}(z)$ (Lemma 5.5 in \cite{Antti}). In particular $h'_{t}(z)$ and $\hat{f}'_{t}(z)$ have the same law, hence
\begin{equation*}
\mathbb{E}\left[|\hat{f}'_{t}(iy)|^p\right] = \mathbb{E}\left[|h'_{t}(iy)|^p\right].
\end{equation*}

The derivation of the upper bound is based on a certain martingale for the reverse flow, see \cite{Optimal}. For this let us introduce a little bit of notation: set $Z_t = h_t - \sqrt{\kappa}B_t$ and $Y_t = \text{Im } h_t$.
\begin{lemma}[Theorem 5.5 in \cite{Antti}]\label{lemma_martimagle}
The stochastic process $\left\{M_t\right\}_{t\ge 0}$ given by
\begin{equation*}
M_t= |h_{t}'\left(z\right)|^p Y_{t}^{p-\frac{\kappa r}{2}}\left(\sin\arg Z_{t}\right)^{-2r}
\end{equation*}
is a local martingale if $p$ and $r$ are the locus of $r^2 - \left(1+\frac{4}{\kappa}\right)r+\frac{2}{\kappa}p=0$. 
\end{lemma}

The reverse Loewner flow implies that $Y_{t}\ge y $ and trivially $\left(\sin\arg Z_{t}\right)^{-2r} \ge 1$ for $r\ge0$. Consequently, for any $t\ge0$ the local martingale $M_{t}$ is bounded from below
\begin{equation*}
M_t= |h_{t}'\left(z\right)|^p Y_{t}^{p-\frac{\kappa r}{2}}\left(\sin\arg Z_{t}\right)^{-2r} \ge 0,
\end{equation*}
hence it is supermartingale (see Chapter 7 of \cite{Oksendal}). Now we can use the supermartingale property of $M_{t}$ to deduce 
\begin{equation*}
\mathbb{E}\left[|h_{t}'\left(z\right)|^p \right] \le  \mathbb{E}\left[|h_{t}'\left(z\right)|^p \left(\frac{Y_{t}}{y}\right)^{p-\frac{\kappa r}{2}}\left(\sin\arg Z_{t}\right)^{-2r} \right] \le \left(\frac{y}{|z|}\right)^{-2r}.
\end{equation*}
Note that in the  first inequality we have used the aforementioned properties
\begin{equation*}
1 \le \frac{Y_t}{y} \text{ and } 1\le \left(\sin\arg Z_{t}\right)^{-2r},
\end{equation*}
which yield the restrictions $r\ge0$ and $p-\frac{\kappa r}{2}\ge0$.  Together with the quadratic equation in Lemma \ref{lemma_martimagle}, they provide admissible values of $p$, of which the largest possible one is given by 
\begin{equation*}
p_0 = \begin{cases}
1 + \frac{2}{\kappa} + \frac{\kappa}{8}, \ \kappa<4,\\
2, \ \kappa \ge 4.
\end{cases}
\end{equation*} 
Of course, one can pick any $p$ between zero and the largest value in order for the moment estimate to hold. Since we are interested in small $\kappa$, we choose $p=\frac{2}{\kappa}$ to simplify further expressions. Therefore, we will utilize the following bound
\begin{equation}\label{expectation_bound}
\mathbb{E}\left[|\hat{f}_{t}'\left(iy\right)|^{2/\kappa} \right]  = \mathbb{E}\left[|h_{t}'\left(iy\right)|^{2/\kappa} \right] \le 1.
\end{equation}
\subsection{Dyadic decomposition}\label{kappa_dependence}
In this subsection we derive an upper bound on the probability of the complement of the event
\begin{equation*}
\left\{|\hat{f}'_{t}(iy)| < c y^{-\beta}, y\in \left[0,2^{-n}\right], t\in\left[0,1\right]\right\},
\end{equation*}
where $\hat{f}_{t}\left(z\right) = f_{t}\left(\sqrt{\kappa}B_{t}+z\right)$; $c$ and $\beta \in \left(0,1\right)$  are $\kappa$-independent parameters.

First, recall the Koebe distortion theorem. 
\begin{lemma}[Koebe distortion theorem]\label{Koebe}
Let $f:D\to\mathbb{C}$ be a conformal map from a simply connected region $D$ and set $d=\text{dist}\left(z, \partial D\right)$ for $z\in D$. Then for $r\in \left[0,1\right)$
\begin{equation*}
\frac{1-r}{\left(1+r\right)^3}|f'\left(z\right)| \le |f'\left(w\right)| \le \frac{1+r}{\left(1-r\right)^3}|f'\left(z\right)|, \quad |z-w|\le rd.
\end{equation*}
\end{lemma}

The following proposition originates from \cite{Optimal} but is adapted to our setting  since we need to know how different constants depend on $\kappa$ explicitly. In what follows we use the notation $j=\overline{1,n}$ to denote $j$ ranging over the set $ \left\{1, ..., n\right\}$.

\begin{proposition}[Dyadic decomposition]\label{dyadic_decomposition}
Let $f$ be the inverse Loewner flow driven by $\lambda$. Fix  $\beta \in \left(0,1\right)$ and $n\in \mathbb{N}$. If for any integer $m\ge n$ and $j=\overline{1,2^{2m}}$
\begin{equation*}
|\hat{f}'_{j/2^{2m}}\left(i2^{-m}\right)| \le 2 ^{\beta m},
\end{equation*}	
then for any $y \in \left[0, 2^{-n}\right]$ and $t\in[0,1]$
\begin{equation}\label{definition_of_E}
|\hat{f}'_{t}\left(iy\right)|\le Q(p(t,y))y^{-\beta},
\end{equation}	
where $Q(x) = c_1 \left(1+ x^2\right)^{c_2}$, $p(t,y) = \frac{1}{y}\sup\limits_{s\in[0,y^2]}\left|\lambda(t+s)-\lambda(t)\right|$ and $c_1, c_2$ are universal ($\kappa$-independent) constants.
\end{proposition}
\begin{proof}

Consider a dyadic decomposition of the unit square in the $\left(t,y\right)$-plane into the union of rectangles  $$S_{m,j} = \left\{(t,y): y\in \left[2^{-m}, 2^{-(m+1)}\right), t\in \left[ j 2^{-2m}, (j+1)2^{-2m}\right) \right\}.$$

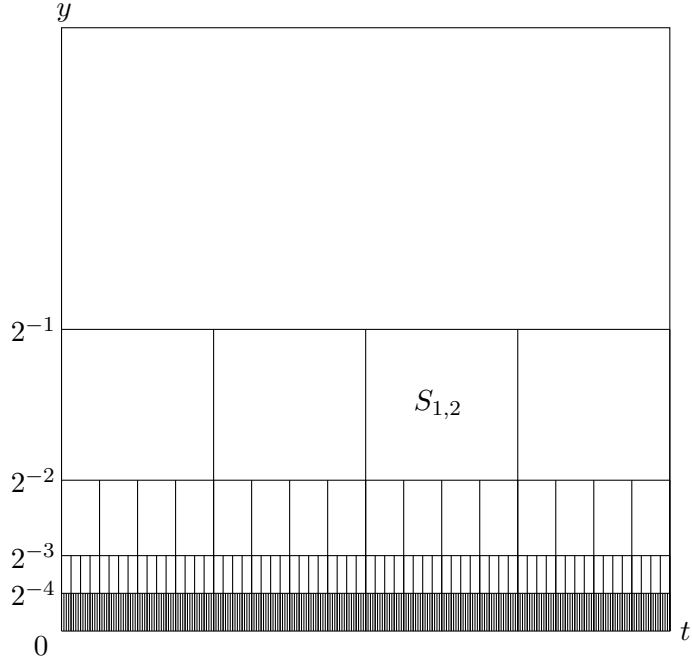
\begin{figure}[H]
\begin{tikzpicture}
\draw (0,0) -- (8,0) -- (8,8) -- (0,8) -- (0,0);
\draw (0,4) -- (8,4);
\draw (0,2) -- (8,2);
\draw (0,1) -- (8,1);
\draw (0,0.5) -- (8,0.5);

\foreach \s in {1,...,4}
{
\draw (8*\s/4,0) -- (8*\s/4,8/2);
}

\foreach \s in {1,...,16}
{
\draw (8*\s/16,0) -- (8*\s/16,8/4);
}

\foreach \s in {1,...,64}
{
\draw (8*\s/64,0) -- (8*\s/64,8/8);
}
\foreach \s in {1,...,256}
{
\draw (8*\s/256,0) -- (8*\s/256,8/16);
}

\filldraw[black] (-0.8,4 )  node[anchor=west]{$2^{-1}$};
\filldraw[black] (-0.8,2 )  node[anchor=west]{$2^{-2}$};
\filldraw[black] (-0.8,1 )  node[anchor=west]{$2^{-3}$}; 
\filldraw[black] (-0.8,0.5 )  node[anchor=west]{$2^{-4}$};
\filldraw[black] (-0.5,-0.2)  node[anchor=west]{$0$};
\filldraw[black] (8,0)  node[anchor=west]{$t$};
\filldraw[black] (-0.2,8.2)  node[anchor=west]{$y$};
\filldraw[black] (4.5,3)  node[anchor=west]{$ S_{1,2}$};
\end{tikzpicture}
\centering
\caption[]{Dyadic decomposition of the unite square.}
\end{figure}

Distortion theorem allows us to estimate the derivative $|\hat{f}_{t}'(iy)|$ for $(t,y)\in S_{m,j}$ by its value at the corner of a rectangle $S_{m,j}$, i.e., 
\begin{equation}\label{rectangle_estimate}
|\hat{f}_{t}'(iy)| \lesssim |\hat{f}_{j/2^{2m}}'(i2^{-m})| \text{ for }(t,y) \text{ in } S_{m,j}.
\end{equation}
 The inequality (\ref{rectangle_estimate}) is obtained in three steps: for $(t,y)\in S_{m,j}$ we derive it in the following order
\begin{equation}\label{chain_of_inequalitites}
|f'_{t}(\lambda_t+iy)| \stackrel{(a)}{\lesssim}  |f'_{t}(\lambda_t+i2^{-m})| \stackrel{(b)}{\lesssim} |f'_{j/2^{2m}}(\lambda_t+i2^{-m})| \stackrel{(c)}{\lesssim} |f'_{j/2^{2m}}(\lambda_{j/2^{2m}}+i2^{-m})|.
\end{equation}

\textbf{a)}  Applying the Koebe distortion theorem, Lemma~\ref{Koebe},  to the conformal map $\hat{f}_{t} $ yields 
\begin{equation}\label{f_along_y}
|\hat{f}'_{t}\left(iy\right)|\le 12 |\hat{f}'_{t}\left(i2^{-m}\right)| \text{ for } y\in \left[2^{-\left(m+1\right)}, 2^{-m}\right].
\end{equation}

\textbf{b)} Next we obtain a bound $|f'_{t+s}\left(z\right)| \lesssim |f'_{t}\left(z\right)|$;
note that here we work with $f$, not with the centered map $\hat{f}$. Expand $f_{t}\left(\xi\right)$ around $z$ to get
\begin{equation*}
f_{t}\left(\xi\right) = f_{t}\left(z\right) + f_{t}'\left(z\right)\left(\xi-z\right) + \frac{1}{2}f''_t(z)\left(\xi-z\right)^{2} +  o\left(\left(\xi-z\right)^{2}\right).
\end{equation*}
We can rearrange this expression to obtain a univalent function of canonical form. Define $F_t:\mathbb{D}\to \mathbb{C}$ by setting
\begin{equation*}
F_t(w) = \frac{f_{t}\left(yw+z \right)-f_{t}\left(z\right)}{f'_{t}\left(z\right)y}.
\end{equation*}
It has the following expansion around the origin 
\begin{equation*}
F_t(w)= w + \frac{f''_{t}y}{2f'_{t}\left(z\right)}w^{2}  + o\left(w^{2}\right).
\end{equation*} 
That is $F_t(0)=0$ and $F_t'(0)=1$, so we can apply Bieberbach's conjecture (de Branges' theorem) which tells that the absolute value of the second coefficient is bounded by two
\begin{equation}\label{f''}
\left|\frac{f''_{t}\left(z\right)y}{2f'_{t}\left(z\right)}\right| \le 2.
\end{equation} 
We are going to use it to obtain a bound on $\left|\partial_{t}f_{t}'\right| $. For that differentiate the Loewner equation (\ref{Loewner_equation_for_f}) with respect to z 
\begin{equation*}
\partial_{t}f_{t}'= - \frac{2f''_{t}}{z-\lambda_{t}} + \frac{2f'_{t}}{\left(z-\lambda_{t}\right)^{2}},
\end{equation*}  
which gives us an estimate on the time-derivative 
\begin{equation*}
\left|\partial_{t}f_{t}'\right|\le  \left|\frac{2f''_{t}}{z-\lambda_{t}} \right| + \left|\frac{2f'_{t}}{\left(z-\lambda_{t}\right)^{2}}\right|.
\end{equation*}
After substituting (\ref{f''}) and noticing that $|z-\lambda_{t}|\ge \text{Im}z \equiv y$ the inequality becomes
\begin{equation*}
 \left|\partial_{t}f_{t}'\right| \le \frac{10 |f'_{t}|}{y^{2}}.
\end{equation*}
It allows us to estimate the following logarithmic derivative
\begin{equation*}
\partial_t\log|f'_t|\le \frac{|\partial_t f'_t|}{|f'_t|} \le \frac{10}{y^2}.
\end{equation*}
Integration over $[t,t+s]$ yields
\begin{equation*}
|f'_{t+s}\left(z\right)|\le e^{\frac{10s}{y^{2}}} |f'_{t}\left(z\right)|.
\end{equation*}

In our application of this inequality, when $(t+s, y)$ lives inside a given rectangle of the dyadic decomposition,  $s$ varies inside $[0,y^2]$, so we can simply use
\begin{equation}\label{func_at_different_times}
|f'_{t+s}\left(z\right)| \le e^{10} |f'_{t}\left(z\right)|.
\end{equation}

\textbf{c)} Next we would like to obtain an inequality $|f_{t}'\left(\lambda_{t+s}+iy\right)| \lesssim |f_{t}'\left(\lambda_{t}+iy\right)|$, where the real parts of the spatial arguments differ. Here we also restrict $s$ to the interval $\left[0,y^{2}\right]$ since it is the case in the dyadic decomposition. The family of points $ \left\{\lambda_{t+s}+iy\right\}_{s\in [0,y^2]} $ is restricted to the semi-disk around $\lambda_t$ of radius
 \begin{equation*}
 R(t,y) =  \sqrt{y^2+\left(\sup\limits_{s\in[0,y^2]} |\lambda_{t+s}-\lambda_{t}|\right)^2  }.
 \end{equation*}

Define a function $g\left(w\right) = f_{t}\left(\lambda_{t} + Rw\right)$ and a rectangle $S = \left\{ x+iy \in \mathbb{C}: x\in \left[-1,1\right], y\in\left[0,1\right]\right\}$. Then  $g:2S \to \mathbb{C}$ is a conformal map. If we consider two points 
\begin{equation*}
w_{\tau} = \frac{\lambda_{\tau}-\lambda_{t}+iy}{R}, \quad \tau = t, t+s, 
\end{equation*}
then both $w_{t}, w_{t+s}\in S$ and Im$w_t=$ Im$w_{t+s} = \frac{y}{R}$. Adjusting Lemma~\ref{Koebe} to the geometry of the domain $S$, see Lemma \ref{distortion_in_S} in the Appendix for the proof, we obtain the  following estimate
\begin{equation*}
|g'\left(w_{t+s}\right)| \le \tilde{c}_1 \left(\frac{R(t,y)}{y}\right)^{\tilde{c}_{2}} |g'\left(w_{t}\right)|,
\end{equation*}
where $\tilde{c_1}, \tilde{c_2}$ are universal constants (do not depend on $\kappa$). Hence, for $s\in \left[0,y^2\right]$, the inequality can be rewritten in the original notation as  
\begin{equation}\label{f_at_different_points}
|f_{t}'\left(\lambda_{t+s}+iy\right) |\le \tilde{c}_1  \left(\frac{R(t,y)}{y}\right)^{\tilde{c}_2} |f_{t}'\left(\lambda_{t}+iy\right)| 
\end{equation}

Combining (\ref{f_along_y}), (\ref{func_at_different_times}) and (\ref{f_at_different_points}) together gives us the chain (\ref{chain_of_inequalitites}) with missing constants. That is,  for $y\in \left[2^{-\left(m+1\right)},2^{-m}\right]$ and $t\in \left[j/2^{2m}, (j+1)/2^{2m}\right]$
\begin{equation}
|\hat f'_{t}(iy)| \le 12e^{10} \tilde{c_1}\left(1 + \left(\frac{1}{y} \sup\limits_{s\in[0,y^2]} |\lambda(t+s) - \lambda(t)|\right)^2\right)^{\tilde{c}_2/2} |\hat f'_{j/2^{2m}}(i2^{-m})|.
\end{equation}
The claim follows by setting $c_1 = 12e^{10}\tilde{c_1}$ and $c_2 = \tilde{c_2}/2$.
\end{proof}

Although the lemma we have just proved holds for any Loewner map, we are interested in its application to SLE. Returning to the convention when $f$ denotes the inverse Loewner map driven by $\sqrt{\kappa}B$,  we introduce the following event
\begin{equation}\label{definition_E}
 E_n \overset{\mathrm{def}}{=} \left\{|\hat{f}'_{t}\left(iy\right)|\le Q(p(t,y)) y^{-\beta} \text{ for all } y\in \left[0, 2^{-n}\right], \ t\in \left[0,1\right]\right\},
\end{equation}
where
\begin{equation*}
p(t,y) = \frac{\sqrt{\kappa}}{y}\sup\limits_{s\in[0,y^2]}|B(t+s) - B(t)|.
\end{equation*}

Proposition \ref{dyadic_decomposition} allows to give an upper bound on the probability of the complement event.
\begin{proposition}\label{proposition_E_compliment}
For fixed $\beta \in (0,1)$ and $\kappa < \beta$  we have 
\begin{equation}\label{E_complement} 
\mathbb{P}\left[E^{c}_n\right] \le   \frac{4^n}{1- 4^{-\left(\frac{\beta}{\kappa}-1\right)}} 4^{-\frac{\beta n }{\kappa}}.  
\end{equation}
\end{proposition}

\begin{proof}
By Proposition \ref{dyadic_decomposition} and a union bound
\begin{equation*}
\begin{split}
\mathbb{P}\left[E_n^c\right] 
&\le \mathbb{P}\left[\bigcup_{m=n}^{\infty}\left\{|\hat{f}'_{j/2^{2m}}\left(i2^{-m}\right)| \le 2 ^{\beta m} \text{ for all } j=\overline{1,2^{2m}}\right\}^{c}\right] \\
&\le\sum\limits_{m=n}^{\infty} \mathbb{P}\left[\exists j \in \overline{1,2^{2m}}:|\hat{f}'_{j/2^{2m}}\left(i2^{-m}\right)| > 2 ^{\beta m}\right]\\
&\le\sum\limits_{m=n}^{\infty}\sum\limits_{j=1}^{2^{2m}} \mathbb{P}\left[|\hat{f}'_{j/2^{2m}}\left(i2^{-m}\right)| > 2 ^{\beta m}\right].
\end{split}
\end{equation*}
By the Chebyshev inequality
\begin{equation*}
\mathbb{P}\left[|\hat{f}'_{j/2^{2m}}\left(i2^{-m}\right)| > 2 ^{\beta m}\right] \le 2^{-\frac{2m \beta}{\kappa}}\mathbb{E}\left[\left|\hat{f}'_{j/2^{2m}}\left(i2^{-m}\right)\right|^{2/\kappa}\right].
\end{equation*}
 The expectation is bounded above by $1$ due to the  moment estimate  (\ref{expectation_bound}) and, then, the summation becomes
\begin{equation*}
\sum\limits_{m=n}^{\infty}\sum\limits_{j=1}^{2^{2m}} \mathbb{P}\left[|\hat{f}'_{j/2^{2m}}\left(i2^{-m}\right)| \ge 2 ^{\beta m}\right] \le \sum\limits_{m=n}^{\infty} 2^{-2\left(\frac{\beta}{\kappa}-1\right) m} = \frac{4^{-\left(\frac{\beta}{\kappa}-1\right) n}}{1- 4^{-\left(\frac{\beta}{\kappa}-1\right)}} \text{ for } \kappa < \beta.
\end{equation*}
\end{proof}
In the next section we will need analogous result for a similar event but using a different $y$-scale, namely
\begin{equation}\label{definition_F}
F_n \overset{\mathrm{def}}{=}  \left\{|\hat{f}'_{t}\left(iy\right)| \le Q(p(t,y)) y^{-\beta} \text{ for } y\in \left[0,1/\sqrt{n}\right], \ t\in\left[0,1\right]\right\}.
\end{equation}
Compare with the definition (\ref{definition_E}) of the event $E_n$. The upper bound of the probability of the compliment follows immediately from Proposition \ref{proposition_E_compliment} by changing the scale from $2^{-n}$ to $1/\sqrt{n}$
\begin{equation}\label{compliment_F}
\mathbb{P}\left[F^{c}_n\right] \le  \frac{n}{1- 4^{-\left(\frac{\beta}{\kappa}-1\right)}} e^{-\frac{\beta\log n }{\kappa}}.
\end{equation}

We would like to have an estimate $|\hat{f}'_{t}\left(iy\right)|\le \psi(n) y^{-\beta}$ where $\psi$ is $n$-dependent function and not a random variable like in (\ref{definition_E}) or (\ref{definition_F}). 
For that we define 
\begin{equation}\label{definition_P_H}
\begin{split}
&P_{n} \overset{\mathrm{def}}{=}  \left\{|\hat{f}_{t}(iy)|\le \psi(n)y^{-\beta} \text{ for } y\in[0, 2^{-n}], t\in[0,1] \right\},\\
&H_{n} \overset{\mathrm{def}}{=}  \left\{|\hat{f}_{t}(iy)|\le \psi(n)y^{-\beta} \text{ for } y\in[0, 1/\sqrt{n}], t\in[0,1] \right\}, 
\end{split}
\end{equation} 
where the $n$-dependent factor is given by
\begin{equation*}
\psi(n) = c_1\left(1+\log n\right)^{c_2};
\end{equation*}
recall that the constants $c_1$ and $c_2$ are universal, i.e., do not depend on $\kappa$. The events $P_n$ and $E_n$ (similarly $F_n$ and $H_n$) satisfy the following inclusion relation
\begin{equation*}
 E_{n}\cap \left\{\sup\limits_{y\in [0,2^{-n}]} \text{osc}\left(B/y, y^2, [0,1]\right) < \sqrt{\frac{\log n}{\kappa}}\right\}  \subset P_n;
\end{equation*}
here we denote by 
  \begin{equation*}
 \text{osc}\left(\lambda, \delta, [0,T]\right)\overset{\mathrm{def}}{=}  \sup \left\{|\lambda(t) - \lambda(s)|: t,s \in [0,T], |t-s|\le \delta\right\} 
  \end{equation*} 
the oscillation of the function $\lambda$ over the interval $[0,T]$ with the increment $\delta$. This inclusion together with Proposition \ref{proposition_E_compliment} allow us to estimate the probability of the complement of $P_n$
\begin{equation}\label{complement_P}
\mathbb{P}\left[P_n^c\right] \le \mathbb{P}\left[E_n^c\right] + \mathbb{P}\left[\sup\limits_{s\in[0,1]} |B(s)| \ge \sqrt{\frac{\log n}{\kappa}}\right] \le \frac{4^n}{1- 4^{-\left(\frac{\beta}{\kappa}-1\right)}} 4^{-\frac{\beta n }{\kappa}} + 2 e^{-\frac{\log n }{\kappa}}.
\end{equation}

The term with Brownian motion is bounded with the help of Brownian scaling
\begin{equation*}
\text{osc}\left(\frac{B}{y}, y^2, [0,1]\right) \overset{\mathrm{d}}{=}  \sup\limits_{s\in[0,1]}\left|B(s)\right|
\end{equation*}

 and the following estimate, for $x>0$,
\begin{equation*}
 \mathbb{P}\left[\sup\limits_{s\in[0,1]}\left|B(s)\right|\ge x\right] \le 2 e^{-x^2}.
\end{equation*}

 Estimate analogous to (\ref{complement_P}) but for the event $H_n$ is given by
\begin{equation}\label{compliements_H_F}
 \mathbb{P}\left[H_n^c\right]\le \frac{n}{1- 4^{-\left(\frac{\beta}{\kappa}-1\right)}} e^{-\frac{\beta\log n }{\kappa}}   + 2e^{-\frac{\log n }{\kappa}}.
\end{equation}

\section{Approximation of SLE$_\kappa$}\label{section_convergence_to_SLE}
In this section we will consider curves generated by approximations of Brownian motion. Let $B_n$ denote a piece-wise linear approximation of Brownian motion, i.e.,
	\begin{equation*}
		B_n(t) = n\left( B(t_{k})-B(t_{k-1}) \right)\left(t-t_{k-1}\right) + B(t_{k-1}) \text{ for } t\in [t_{k-1},t_{k}),
	\end{equation*}
where $\{t_{k}\}_{k=0}^{n}$ is a partition of $\left[0,1\right]$ into $n$ equal intervals.  In \cite{Tran} H. Tran showed that $\gamma_{n}= \mathcal{L}\left(\sqrt{\kappa}B_n\right)$ converges almost surely to SLE$_\kappa$ in the supremum norm. The convergence follows from the following estimate obtained in the aforementioned paper
\begin{equation*}
	\mathbb{P}\left[ ||\gamma^\kappa-\gamma_{n}||_{\infty}\le \frac{\varphi(n)}{\sqrt{n}^{1-\sqrt{\frac{1+\beta}{2}}}}\right] \ge 1- \frac{c_3}{n^{c_{4}}},
\end{equation*}
where $\beta \in (0,1)$ and  $\varphi$ is a sub-power function and $c_3, c_4$ are $\kappa$-dependent constants. Recall that $\varphi: [0,\infty) \to [0,\infty)$ is called a sub-power function if for every $\alpha>0$
\begin{equation*}
\lim\limits_{x\to\infty}\frac{\varphi(x)}{x^\alpha} = 0.
\end{equation*}

In our application we need to know explicitly how the constants $c_3, c_4$ and the sub-power function $\varphi$ depend on $\kappa$.  In the original proof the parameter $\kappa$ was fixed since the emphasis was on the behavior in the limit $n\to\infty$.  Our situation is the opposite, and we will be interested in the behavior when $n$ is fixed while $\kappa \downarrow 0$.

\begin{proposition}\label{convergence_result}
Fix $\beta \in (0,1)$. Then for any $\kappa<\beta$ and $ \zeta \in \left(0, \frac{1}{2}\left(1-\sqrt{\frac{1+\beta}{2}}\right)\right)$ there exists $N=N(\beta, \zeta)\in\mathbb{N}$ such that for any integer $n>N$
\begin{equation}\label{convergence_bound}
\mathbb{P}\left[||\gamma-\gamma_n||_{\infty} \ge n^{-\zeta}\right] \le B(n,\kappa)\left(n/2\right)^{-\frac{\beta}{\kappa}},
\end{equation}
where
\begin{equation}\label{definition_B}
B\left(n,\kappa\right) = 2 + c_0 +\frac{n}{1 - 4^{-\left(\frac{\beta}{\kappa}-1\right)}}
\end{equation}
and $c_0$ is universal constant.
\end{proposition}

Examining the proof in \cite{Tran} we see that the inequality  \[||\gamma-\gamma_n||_{\infty} \le n^{-\zeta}\] holds on the intersection of the following two events:\\
1) $H_n = \left\{|\hat{f}'_{t}(iy)| \le \psi(n) y^{-\beta} \text{ for } y\in \left[0,1/\sqrt{n}\right], t\in\left[0,1\right]\right\}$, with $\psi(n) = c_1\left(1+\log n\right)^{c_2}$;\\
2) $L_n = \left\{ \text{osc}\left(\sqrt{\kappa}B, 2/n, [0,1]\right) \le \sqrt{2/n}\varphi_L\left(n/2\right)\right\}$, where $\varphi_L$ is a sub-power function.\\
The estimate (\ref{compliements_H_F}) covers the needed estimate for the first event while the second one can be deduced from the following lemma.
\begin{lemma}[Theorem 3.2.4 in \cite{Lawler_Limic}]\label{Lawler_Limic_lemma}
Let B be the standard Brownian motion on $[0,1]$. There is an absolute constant $c_0<\infty$ such that for all $0<\delta\le 1$ and $r>c_0$
\begin{equation*}
\mathbb{P}\left[\text{osc}\left(\sqrt{\kappa}B(t), \delta, [0,1]\right) \ge r \sqrt{\delta\log\left(1/\delta\right)}\right]\le c_0 \delta^{\frac{(r/c_0)^2}{\kappa}}.
\end{equation*}
\end{lemma}
For our application pick $\delta = 2/n$ and, for example, $r=\sqrt{2}c_0$.  Then the sub-power function $\varphi_L(x) = c_0\sqrt{2\log x}$ gives the second assumption with the probability estimate given by
\begin{equation*}
\mathbb{P}\left[L_n^c\right] \le c_0 \left(n/2\right)^{-\frac{2}{\kappa}}.
\end{equation*}
A union bound then implies
\begin{equation*}
\begin{split}
\mathbb{P}\left[||\gamma-\gamma_n||_{\infty} \ge n^{-\zeta}\right] &\le \mathbb{P}\left[H_n^c\right] + \mathbb{P}\left[L_n^c\right]\\
& \le \frac{n}{1- 4^{-\left(\frac{\beta}{\kappa}-1\right)}} n^{-\frac{\beta}{\kappa}}   + 2n^{-\frac{1}{\kappa}} + c_0 \left(n/2\right)^{-\frac{2}{\kappa}},
\end{split}
\end{equation*} 
which gives estimate (\ref{convergence_bound}).

\section{Large deviation principle: proof of Theorem~\ref{LDP_theorem}}\label{section_proof}
This section proves Theorem~\ref{LDP_theorem}. The proof is divided into three parts, which are completed in its own subsection: first we derive the theorem for a finite time interval $[0,T]$, where closed and open sets are treated separately, and afterwards extend obtained result to the time interval $[0,\infty)$. As mentioned in the outline of the proof, for finite time interval it is sufficient to consider the  interval $[0,1]$ due to scale invariance of SLE.

\subsection{Closed sets}
In this section $\gamma$ is SLE$_\kappa$ and $\gamma_n$ is the Loewner curve driven by $\sqrt{\kappa}B_n$, where $B_n$ is a piece-wise linear approximation of Brownian motion.

Let $F$ be a closed subset in $S_1$; recall the definition of this space in (\ref{definition_S_T}). For any $\delta > 0$ and $n\in \mathbb{N}$ consider the following decomposition  
	\begin{equation}\label{decomposition_for_closed_case}
		\mathbb{P}\left[\gamma \in F\right] = \mathbb{P}\left[\gamma\in F, ||\gamma-\gamma_{n}||_{\infty} < \delta \right] + \mathbb{P}\left[\gamma\in F, ||\gamma-\gamma_{n}||_{\infty} \ge \delta \right].
	\end{equation}
	The main idea is to show that in the limit $\kappa \downarrow 0 $ the first term is dominant if we tune the parameters $n$ and $\delta$ appropriately. 
	
	\textbf{(a)} Denote by $F^{\delta}$ the $\delta$-neighborhood of $F$. Then, we have
	 \begin{equation*}
		\mathbb{P}\left[\gamma\in F, ||\gamma-\gamma_{n}||_{\infty} < \delta \right] \le \mathbb{P}\left[\gamma_{n} \in F^{\delta}\right] \le \mathbb{P}\left[I\left(\gamma_n\right) \ge I\left(F^{\delta}\right)\right] .
	\end{equation*}
The Loewner energy of the curve $\gamma_n$ can be computed explicitly due to piece-wise linearity of the driving function. We have	
	\begin{equation*}
		\frac{2}{\kappa}I(\gamma_n) = \frac{2}{\kappa} I\left(\sqrt{\kappa} B_n\right) = \sum\limits_{k=1}^{n}n\left(B(t_{k}) - B(t_{k-1})\right)^{2}.
	\end{equation*}
Standard properties of Brownian motion imply that the terms of this sum are mutually independent and normally distributed. Hence, the sum has a $\chi^{2}$-distribution which allows us to give the following estimate, valid for $\kappa < I(F^\delta)/n$,
	\begin{equation*}
		\begin{split}
			\mathbb{P}\left[  I\left(\gamma_n\right) \ge I(F^{\delta}) \right]  = \frac{1}{2^{\frac{n}{2}}\Gamma(\frac{n}{2})} \int\limits_{\frac{2I(F^{\delta})}{\kappa}}^{\infty}x^{\frac{n}{2}-1}e^{-\frac{x}{2}}dx\le \frac{4\left(\frac{2I(F^{\delta})}{\kappa}\right)^{\frac{n}{2}-1}}{2^{\frac{n}{2}}\Gamma(\frac{n}{2})}e^{-\frac{I(F^{\delta})}{\kappa}}.
		\end{split}
	\end{equation*}
	This provides the upper bound on the first term in (\ref{decomposition_for_closed_case}):
	\begin{equation}\label{a_term}
	\mathbb{P}\left[\gamma\in F, ||\gamma-\gamma_{n}||_{\infty} < \delta \right] \le A\left(\kappa,n\right)e^{-\frac{I(F^{\delta})}{\kappa}}, \text{ where } A(\kappa, n) = \frac{4\left(\frac{2I(F^{\delta})}{\kappa}\right)^{\frac{n}{2}-1}}{2^{\frac{n}{2}}\Gamma(\frac{n}{2})}.
	\end{equation}
	
	\textbf{(b)} The second term in (\ref{decomposition_for_closed_case}) can be first bounded by trivial inclusion 
	\begin{equation*}
	\mathbb{P}\left[\gamma\in F, ||\gamma-\gamma_{n}||_{\infty} \ge \delta \right] \le \mathbb{P}\left[||\gamma-\gamma_{n}||_{\infty} \ge \delta \right].
	\end{equation*}
	Next, we use convergence of the approximating curve to SLE, see Section \ref{section_convergence_to_SLE}. Proposition \ref{convergence_result} states that for any $\beta \in \left(0,1\right)$, any $\kappa < \beta$ and $ \zeta \in \left(0, \frac{1}{2}\left(1-\sqrt{\frac{1+\beta}{2}}\right)\right)$ there exists $N=N(\beta, \zeta)\in\mathbb{N}$ such that for all $n>N$ 
	\begin{equation*}
		\mathbb{P}\left[ ||\gamma-\gamma_{n}||_{\infty}\ge n^{-\zeta}\right] \le B\left(\kappa, n\right)e^{-\frac{\beta\log (n/2)}{\kappa}}, \text{ where }B(\kappa, n) = 2 + c_0 +\frac{n}{1 - 4^{-\left(\frac{\beta}{\kappa}-1\right)}}.
	\end{equation*}
If we choose $n=n(\delta, \zeta, \beta)$ such that $\delta > n^{-\zeta}$, then the second term in (\ref{decomposition_for_closed_case}) is bounded by 

	\begin{equation}\label{b_term}
\mathbb{P}\left[\gamma\in F, ||\gamma-\gamma_{n}||_{\infty} \ge \delta \right]\le B\left(\kappa, n\right) e^{-\frac{\beta\log (n/2)}{\kappa}}.
	\end{equation}
Therefore, 	combining (\ref{a_term}) and (\ref{b_term}) together gives
\begin{equation*}
\mathbb{P}\left[\gamma\in F \right] \le A\left(\kappa,n\right)e^{-\frac{I(F^{\delta})}{\kappa}} +  B\left(\kappa, n\right) e^{-\frac{\beta\log (n/2)}{\kappa}},
\end{equation*}	
and after taking the logarithm of both sides and multiplying by $\kappa$ we obtain
\begin{equation*}
\kappa\log\mathbb{P}\left[\gamma\in F \right]  \le -I\left(F^{\delta}\right) + \kappa \log A\left(\kappa,n\right) + \kappa \log \left(1 + \frac{B\left(\kappa,n\right)}{A\left(\kappa,n\right)}e^{-\frac{\beta\log (n/2)- I\left(F^{\delta}\right)}{\kappa}}\right).
\end{equation*}
If $n$ is big enough, that is, if $n > 2e^{\frac{I(F)}{\beta}}$, then in the limit $\kappa\downarrow 0$ the last two terms will vanish since
\begin{equation*}
\lim\limits_{\kappa\downarrow 0}\frac{B\left(\kappa,n\right)}{A\left(\kappa,n\right)} = \lim\limits_{\kappa\downarrow 0}\frac{\left(2 + c_0 +\frac{n}{1 - 4^{-\left(\frac{\beta}{\kappa}-1\right)}}\right)2^{\frac{n}{2}}\Gamma(\frac{n}{2})}{4\left(\frac{2I(F^{\delta})}{\kappa}\right)^{\frac{n}{2}-1}} = 0
\end{equation*}
and
\begin{equation*}
\lim\limits_{\kappa\downarrow 0}\kappa \log A(\kappa,n) = \lim\limits_{\kappa\downarrow 0}\left(\kappa\log \frac{4\left(2I(F^{\delta})\right)^{\frac{n}{2}-1}}{2^{\frac{n}{2}}\Gamma(\frac{n}{2})} - \left(\frac{n}{2}-1\right)\kappa\log\kappa\right) = 0.
\end{equation*}
Therefore, in the limit
	
	\begin{equation*}
		\uplim\limits_{\kappa\downarrow 0} \kappa \log\left(\mathbb{P}\left[\gamma \in F\right] \right) \le - I(F^{\delta}).
	\end{equation*}
Lemma \ref{closed_sets_rate_function} asserts that taking the limit $\delta \downarrow 0$ gives us the first part of Theorem \ref{LDP_theorem} for finite  time, namely
\begin{equation*}
		\uplim\limits_{\kappa\downarrow 0} \kappa \log\left(\mathbb{P}\left[\gamma \in F\right] \right) \le - I(F) \text{ for any closed  } F \subset S_1.
\end{equation*}

\subsection{Open sets}
Let $G$ be an open set in $S_1$ and assume $I(G) < \infty$, otherwise the LDP inequality is trivial. For every $\alpha > 0$ there is a curve $\gamma^{\alpha}\in G$  such that $I(\gamma_{\alpha})\le I(G)+\alpha$ and, since the set is open,  $G$ contains an open ball $B\left(\gamma_{\alpha}, r_{\alpha}\right)$ around $\gamma_\alpha$ with some positive radius $r_{\alpha}$. Thus, if $\gamma$ denotes SLE$_\kappa$, by monotonicity 
\begin{equation*}
\mathbb{P}\left[\gamma \in G \right] \ge \mathbb{P}\left[ ||\gamma- \gamma_{\alpha}||_{\infty} < r_{\alpha}\right].
\end{equation*}

For $y>0$ consider the following decomposition
\begin{equation*}
|\gamma(t)-\gamma^{\alpha}(t)| \le |\gamma(t)-\hat{f}_{t}(iy)|+ |\gamma_{\alpha}(t)-\hat{f}^{\alpha}_{t}(iy)|+ |\hat{f}_{t}(iy)-\hat{f}^{\alpha}_{t}(iy)|,
\end{equation*}
where $f_{t}$ and $f^{\alpha}_{t}$ are solutions to the Loewner equation with driving functions $\sqrt{\kappa}B$ and $\lambda^{\alpha} = \mathcal{L}^{-1}\left(\gamma^{\alpha}\right)$ correspondingly. The middle term is deterministic and can be bounded from above with the help of Lemma \ref{FrizShekhar_lemma}

\begin{equation*}
|\gamma_{\alpha}(t)-\hat{f}^{\alpha}_{t}(iy)|\le  \int\limits_{0}^{y}|(\hat{f}^\alpha_{t})'(is)| ds\le y e^{\frac{1}{2}I(\gamma_{\alpha})} .
\end{equation*}

Denote  $ R\left(\alpha, y\right) = \frac{1}{2}\left(r_{\alpha} - y e^{\frac{1}{2}I\left(\gamma^\alpha\right)}\right)$ which is positive if $y < y_0 = r_\alpha e^{-\frac{1}{2}I\left(\gamma_\alpha\right)}$. Then,  we can decompose the sum of the two remaining terms as
\begin{equation*}
\begin{split}
&\mathbb{P}\left[||\hat{f}_{t}(iy)-\hat{f}^{\alpha}_{t}(iy)||_{\infty} + ||\gamma (t)-\hat{f}_{t}(iy)||_{\infty} > 2R\right]\\ &\le \mathbb{P}\left[||\hat{f}_{t}(iy)-\hat{f}^{\alpha}_{t}(iy)||_{\infty}  > R\right] + \mathbb{P}\left[||\gamma (t)-\hat{f}_{t}(iy)||_{\infty} > R\right].
\end{split}
\end{equation*}
So, we have
\begin{equation}\label{after_union_bound}
\begin{split}
\mathbb{P}\left[||\gamma-\gamma^{\alpha}||_{\infty} < r_{\alpha}\right] \ge \mathbb{P}\left[||\hat{f}_{t}(iy)-\hat{f}^{\alpha}_{t}(iy)||_{\infty} \le R\right] - \mathbb{P}\left[  ||\gamma (t)-\hat{f}_{t}(iy)||_{\infty} \ge R\right].
\end{split}
\end{equation}
The second term can be dealt with if we employ the event $P_n$, defined in (\ref{definition_P}),
\begin{equation*}
P_n = \left\{|\hat{f}'_t(iy)|\le \psi(n)y^{-\beta} \text{ for }y \in [0,2^{-n}], t\in [0,1]\right\}, \quad \psi(n) = c_1(1+\log n)^{c_2}.
\end{equation*}
We decompose the event $A = \left\{||\gamma (t)-\hat{f}_{t}(iy)||_{\infty}\ge R\right\}$ into $A = A \cap P_n +  A \cap P_n^c$, so its probability can be bounded as
\begin{equation*}
\mathbb{P}\left[A\right] \le \mathbb{P}\left[A\cap P_n\right] +\mathbb{P}\left[P_n^c\right].
\end{equation*} 
We claim that for small enough $y$ the first term vanishes. On the event $P_n$
\begin{equation*}
|\gamma(t) - \hat{f}_t(iy)|\le \int\limits_0^y|\hat{f}'_t(is)|ds \le \frac{c}{1-\beta}y^{1-\beta}.
\end{equation*}
If we restrict our attention to $y < \min\left(y_0, y_1, 2^{-n}\right)$ where $y_1$ is the solution to
\begin{equation*}
\frac{c}{1-\beta}y_1^{1-\beta} = \frac{1}{2}\left(r_\alpha - y_1e^{\frac{1}{2}I\left(\gamma^\alpha\right)}\right) = R(\alpha, y_1),
\end{equation*}
then   $ A \cap P_n = \varnothing$. The probability of $P_n^c$ was bounded in (\ref{complement_P}), hence, for $y < \min\left(y_0, y_1, 2^{-n}\right)$ and $\kappa < \beta$ we have
\begin{equation}\label{gamma_minus_f}
\mathbb{P}\left[  ||\gamma (t)-\hat{f}_{t}(iy)||_{\infty}\ge R\right] \le \frac{4^n}{1- 4^{-\left(\frac{\beta}{\kappa}-1\right)}} 4^{-\frac{\beta n }{\kappa}} + 2 e^{-\frac{\log n}{\kappa}}.
\end{equation}

Now we move to the first term in (\ref{after_union_bound}).  The inequality for the Loewner maps can be translated to the corresponding driving functions due to Lemma \ref{lemma_continuity}. Namely, we have deterministic inequality
\begin{equation}\label{f_to_drivers}
||\hat{f}_{t}(iy)-\hat{f}^{\alpha}_{t}(iy)||_{\infty} \le ||\sqrt{\kappa}B-\lambda^{\alpha}||_{\infty} \sqrt{1+ \frac{4}{y^2}}
\end{equation}
which by monotonicity implies that
\begin{equation*}
\mathbb{P}\left[||\hat{f}_{t}(iy)-\hat{f}^{\alpha}_{t}(iy)||_{\infty} \le R\right] \ge \mathbb{P}\left[||\sqrt{\kappa}B-\lambda^{\alpha}||_{\infty}  \le \frac{R}{\sqrt{1+\frac{4}{y^2}}}\right].
\end{equation*}
The remaining analysis of this term follows that of Schilder's theorem. 
The idea now is to use absolute continuity of $\lambda^\alpha$ and apply Girsanov's theorem to consider $B - \frac{1}{\sqrt{\kappa}}\lambda^\alpha$ as Brownian motion under a new measure. However, in order to control the Radon-Nikodym derivative  we approximate $\lambda^\alpha$ by a smooth function with very close Dirichlet energy. 

For any $\varepsilon \in \left(0,\alpha\right)$ it is possible to find an absolutely continuous  function $\varphi_{\varepsilon} \in C^2\left(\left[0,1\right]\right)$ such that $\varphi_{\varepsilon}(0)=0$ and $I\left(\lambda^\alpha-\varphi_{\varepsilon}\right) < \varepsilon^2/2$. The latter condition ensures closeness of the energies. Since $\sqrt{I(\cdot)}$ is an $L^2\left(\left[0,1\right]\right)$-norm of a derivative, by triangle inequality 
\begin{equation*}
\sqrt{I(\varphi_\varepsilon)}\le \sqrt{I\left(\lambda^{\alpha}\right)} + \sqrt{I\left(\lambda^{\alpha} - \varphi_{\varepsilon}\right)} \le \sqrt{I(\lambda^\alpha)} + \frac{\varepsilon}{\sqrt{2}}.
\end{equation*}
Moreover, for any $t\in \left[0,1\right]$, by the Cauchy-Schwarz inequality 
\begin{equation*}
|\lambda_{\alpha}(t)- \varphi_{\varepsilon}(t)| \le \int\limits_{0}^{t}|\dot{\lambda}_{\alpha}(s)- \dot{\varphi}_{\varepsilon}(s)|ds \le \sqrt{2tI\left(\lambda_{\alpha} - \varphi_{\varepsilon}\right)}\le \varepsilon.
\end{equation*}
Taking the supremum yields $||\lambda_{\alpha} - \varphi_{\varepsilon}||_{\infty} \le \varepsilon$. 
Hence, for any $\varepsilon < \frac{1}{2}\frac{R}{\sqrt{1+\frac{4}{y^2}}}$ and $\delta < \frac{1}{2}\frac{R}{\sqrt{1+\frac{4}{y^2}}}$ the following inclusion holds
\begin{equation*}
B\left(\varphi_{\varepsilon}, \delta\right) \subset  B\left(\lambda^\alpha, \frac{R}{\sqrt{1+\frac{4}{y^2}}} \right), 
\end{equation*}
which by monotonicity implies
\begin{equation*}
\mathbb{P}\left[||\sqrt{\kappa}B-\lambda^{\alpha}||_{\infty} \le \frac{R}{\sqrt{1+\frac{4}{y^2}}}\right] \ge \mathbb{P}\left[||\sqrt{\kappa}B-\varphi_{\varepsilon}||_{\infty} \le \delta\right]. 
\end{equation*}
By the Girsanov theorem,
\begin{equation}\label{new_BW}
\tilde{B}_{t} = B_{t} - \frac{1}{\sqrt{\kappa}}\int\limits_{0}^{t}\dot\varphi_{\varepsilon}(s)ds
\end{equation}
is a standard Brownian motion under the new measure which we denote by $\tilde{\mathbb{P}}$. In our setting it is convenient to apply the Girsanov's theorem the other way around. That is, we start with a Brownian motion $\tilde{B}$ under the probability measure $\tilde{\mathbb{P}}$, then $B$ from (\ref{new_BW}) is a Brownian motion under $\mathbb{P}$. The connection between $\mathbb{P}$ and $\tilde{\mathbb{P}}$, for the event in question, is given by
\begin{equation*}
\mathbb{P}\left[||\tilde{B}||_{\infty} \le \frac{\delta}{\sqrt{\kappa}}\right] = \tilde{\mathbb{E}}\left[\mathbf{1}\left\{||\tilde{B}||_{\infty}\le \frac{\delta}{\sqrt{\kappa}}\right\}\exp\left\{-\frac{1}{\sqrt{\kappa}}\int\limits_{0}^{1}\dot\varphi_{\varepsilon}(s)d\tilde{B}(s)-\frac{1}{2\kappa}\int\limits_{0}^{1}\dot\varphi_{\varepsilon}(s)^2 ds\right\}\right].
\end{equation*}
Next, by integration by parts for stochastic processes the first term in the exponent can be bounded by
\begin{equation*}
\left|\int\limits_{0}^{1}\dot\varphi_{\varepsilon}(s)d\tilde{B}(s)\right| \le h_\varepsilon||\tilde{B}||_{\infty}, \text{ with } h_\varepsilon = 2 \sup\limits_{t\in[0,1]}\left(|\dot\varphi_{\varepsilon}(t)| + |\ddot\varphi_{\varepsilon}(t)|\right).
\end{equation*}
The second term in the exponent is precisely the Dirichlet energy of $\varphi_\varepsilon$, hence 
\begin{equation}\label{f_minus_f}
\mathbb{P}\left[||\hat{f}_{t}(iy)-\hat{f}^{\alpha}_{t}(iy)||_{\infty} \le R\right]   \ge 
e^{-\frac{I(\varphi_{\varepsilon})+h_\varepsilon\delta}{\kappa}}\mathbb{P}\left[||B||_{\infty} \le \frac{\delta}{\sqrt{\kappa}}\right].
\end{equation}
Combining (\ref{f_minus_f})  and (\ref{gamma_minus_f}) in (\ref{after_union_bound}) we obtain

\begin{equation*}
\mathbb{P}\left[\gamma \in G \right] \ge e^{-\frac{I(\varphi_{\varepsilon})+h_\varepsilon\delta}{\kappa}}\mathbb{P}\left[||B||_{\infty} \le \frac{\delta}{\sqrt{\kappa}}\right] - \left(2 +\frac{4^n}{1- 4^{-\left(\frac{\beta}{\kappa}-1\right)}} \right)e^{-\frac{\log n}{\kappa}}.
\end{equation*}
If we choose $n$, so that $\log n > I(\varphi_{\varepsilon})+h_\varepsilon\delta$, then

\begin{equation*}
\lowerlim\limits_{\kappa\downarrow 0} \kappa \log\mathbb{P}\left[\gamma \in G \right] \ge -I(\varphi_\varepsilon)-h(\varepsilon)\delta \ge - \left(\sqrt{I(\lambda_{\alpha})} + \varepsilon\right)^2 - h_\varepsilon\delta. 
\end{equation*}

To conclude take the limits in the following order: first $\delta \downarrow 0$, then $\varepsilon \downarrow 0$ and finally $\alpha \downarrow 0$ to obtain the desired result for open sets
\begin{equation*}
\lowerlim\limits_{\kappa\downarrow 0} \kappa \log \mathbb{P}\left[\gamma\in G\right]\ge -I\left(G\right).
\end{equation*}
This concludes the proof of large deviation principle for the case $T<\infty$.
\subsection{Projective limit}
So far we have proved that LDP holds on a space of capacity parameterized curves in the upper half-plane restricted to finite time interval $[0,T]$ in the topology induced by the supremum norm. This topological space will be denoted by  $\left(S_T, \tau_T\right)$.  Now we move to the space $S$, defined in (\ref{definition_S}), of continuous curves run all the way to infinity. Equip $S$ with the topology $\tau$ of uniform convergence on compact intervals (compact convergence). Our aim is to translate LDP result from $\left(S_T, \tau_T\right)$ to $\left(S, \tau\right)$.

Transition between $\left(S_T, \tau_T\right)$ and $\left(S, \tau\right)$ is carried out using standard tools of Large Deviations theory, namely Dawson-G{\"a}rtner theorem and contraction principle. This route require an intermediate step: construction of \textit{the projective limit} space. So the Dawson-G{\"a}rtner theorem carries LDP result from $\left(S_T, \tau_T\right)$  to the projective limit space and then contraction principle allows to conclude LDP on  $\left(S, \tau\right)$.  Techniques of this section are rather standard and can be found, for example, in \cite{DemboZeitouni}.

Before we construct the projective limit let us first set up necessary notation.   $\gamma_{T}$ denotes restriction of $\gamma \in S$ to the time interval $[0,T]$.  Consequently, $V_T = \left\{\gamma_T: \gamma \in V\right\}$ is the restriction of subset $V\subset S$. Given our family $\left\{S_T\right\}_{T\in\mathbb{R}_+}$ of finite-time spaces we construct \textit{Cartesian product space}
 \begin{equation*}
 S_{\text{car}} = \prod\limits_{T>0}S_{T}.
 \end{equation*}
By definition $ S_{\text{car}} $ is a family of all function $\Gamma: \mathbb{R}_{+}\to \cup_{T>0} S_{T}\text{ such that } \Gamma(T) \in S_{T}$. This space is equipped with \textit{the product topology} generated by a base that consists of sets of the form 
\begin{equation*}
\left\{\Gamma \in  S_{\text{car}}  : \Gamma(T_{k}) \in G_{T_{k}}, k= \overline{1,n}\right\}
\end{equation*} 
where $\{T_{k}\}_{k=1}^{n}$ is a finite subset of $\mathbb{R}_{+}$ and every $G_{T_{k}}$ is an open subset of $S_{T_{k}}$ in the supremum norm topology.

Cartesian product includes many elements that are of no interest to us. In order to restrict our attention we consider a subspace called \textit{the projective limit}, denoted by $\varprojlim S_T$ and defined via the following construction. First, introduce \textit{projective system} $\left(S_t, \pi_{s, t}\right)_{s\le t}$ which consists of a family $\left\{S_s\right\}_{s<t}$ and continuous maps $\pi_{s,t}: S_t\to S_s$ such that $\pi_{s,t} = \pi_{s,\tau }\circ \pi_{\tau, t}$ whenever $s \le \tau \le t $. Then the projective limit $\varprojlim S_T $ is a subspace of $ S_{\text{car}} $ consisting of 
\begin{equation*}
\varprojlim S_T  = \left\{ \left\{\gamma_t\right\}_{t>0} \in S_{\text{car}}:  \gamma_s = \pi_{s,t}(\gamma_t) \text{ for }s\le t\right\}
\end{equation*}
The projective limit is equipped with the topology $\tau_p$ induced by the product topology on $S_\text{car}$.

There is a natural way to construct a projective system. On the product space $ S_{\text{car}}$ we  define coordinate maps $p_T:  S_{\text{car}} \to S_T$ which give $T$-coordinate of $\Gamma \in  S_{\text{car}}$, i.e., $p_T\left[\Gamma\right] = \Gamma(T)$. Restriction of these coordinate maps to the projective limit $\varprojlim S_T$ are called \textit{canonical projections}. These projections $p_T:\varprojlim S_T \to S_T $ are continuous. Note that the product topology  $\tau_p$ is smallest one which ensures continuity of canonical projections. Setting $\pi_{s,t} = p_s\circ p^{-1}_t$ gives rise to projective system since
\begin{equation*}
\pi_{s,t} = p_s \circ p^{-1}_{\tau}\circ p_\tau\circ p^{-1}_t = \pi_{s, \tau}\circ \pi_{\tau, t}, \text{ whenever }  s \le \tau \le t.
\end{equation*}

Having defined the projective limit space $\left(\varprojlim S_T, \tau_p\right)$ we can state our plan in the following pictorial way:

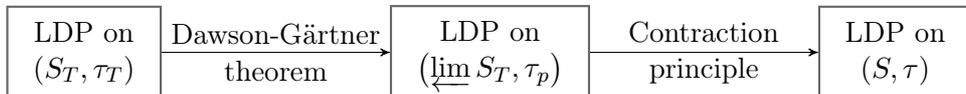
\begin{figure}[H]
\centering
\begin{tikzpicture}
[
node distance=3cm,
squarednode/.style={rectangle, draw=black!60, thick, minimum size=10mm}
]
\node[squarednode]   (ST)  {\begin{tabular}{c} LDP on \\ $\left(S_T, \tau_T\right)$\end{tabular}};

\node[squarednode]  (Sproj)  [right=of ST] {\begin{tabular}{c} LDP on\\  $\left(\varprojlim S_T, \tau_p\right)$\end{tabular}};

\node[squarednode]   (S)  [right=of Sproj] {\begin{tabular}{c} LDP on \\ $\left(S, \tau\right)$\end{tabular}};

\draw[->]  (ST) to node[midway,align=center] { Dawson-G$\ddot{\text{a}}$rtner\\ theorem} (Sproj);
\draw[->]  (Sproj.east) to node[midway,align=center] {Contraction\\ principle} (S.west);
\end{tikzpicture}
\end{figure}

In order to apply contraction principle we need to ensure existence of continuous mapping $p:\varprojlim S_T \to S$. In fact, there is a continuous bijection between the two spaces.  For any $\Gamma \in  \varprojlim S_T$ define a curve $p\left[\Gamma\right]$ by setting $p\left[\Gamma\right](t) = p_t\left[\Gamma\right](t)$ via canonical projections. Conversely, for any $\gamma \in S$ define an element $ p^{-1}\left[\gamma\right] = \left\{\gamma_T\right\}_{T>0}$ of the projective limit space via restrictions. 
\begin{proposition}
A bijection $p: \varprojlim S_T \to S$ is continuous if $S$ is equipped with a topology of compact convergence.  
\end{proposition}
\begin{proof}
For any closed subset $F\subset S$ in the topology of compact convergence 
\begin{equation*}
p^{-1}\left[F\right] = \bigcap_{T>0}p^{-1}_T\left[F_T\right].
\end{equation*}
By continuity of canonical projections every $p^{-1}_T\left[F_T\right]$ is closed and so is arbitrary intersection of closed sets.
\end{proof}

In what follows $\mathbb{P}$ is SLE probability measure and $\gamma^{\kappa}$ is SLE$_\kappa$ curve. Let $\left\{\mu_\kappa\right\}_{\kappa>0}$ be a family of probability measures on measurable space $\left(S, \sigma\left(\tau\right)\right)$, where $\sigma(\tau)$ denotes Borel $\sigma$-algebra generated by $\tau$, given by
\begin{equation*}
\mu_{\kappa}\left(V\right) = \mathbb{P}\left[\gamma^\kappa \in V \right], \quad V \in \sigma\left(\tau\right).
\end{equation*}
 Similarly let $\nu_\kappa$ be a probability measure on $\left(\varprojlim S_T, \sigma\left(\tau_p\right)\right)$ defined by
\begin{equation*}
\nu_\kappa\left(\mathcal{V}  \right) = \mu_\kappa\left(p\left[\mathcal{V}\right]\right) = \mathbb{P}\left[\gamma^\kappa \in p\left[\mathcal{V}\right]\right], \quad \mathcal{V} \in \sigma\left(\tau_p\right).
\end{equation*}
Then $\nu_\kappa$ gives rise to  probability measures on $S_T$ via canonical projections. For any $V_T \in \sigma(\tau_T)$
\begin{equation*}
\nu_\kappa\circ p^{-1}_T\left(V_T\right) = \mu_\kappa\left( p\left[p^{-1}_T\left[V_T\right]\right] \right) = \mathbb{P}\left[\gamma^\kappa_T \in V_T \right]. 
\end{equation*}
In this notation what we have proved so far is that a family $\left\{\nu_\kappa\circ p_T^{-1}\right\}_{\kappa>0}$ of probability measures on $\left(S_T, \sigma(\tau_T)\right)$ satisfies LDP for every $T>0$ with a good rate function $I_T$ given by
\begin{equation*}
I_{T}\left(\gamma_T\right)= \begin{cases}
\frac{1}{2}\int\limits_{0}^{T}\dot{\lambda}(t)^{2}dt, \text{ if }\lambda\text{ is absolutely continuous},\\
\infty, \text{ otherwise};
\end{cases}
\end{equation*}
here $\lambda$ is the driving function of $\gamma_T$.

The Dawson-G$\ddot{\text{a}}$rtner theorem allows to extend this LDP results the projective limit space $\varprojlim S_T$. 
\begin{theorem}[Dawson-G$\ddot{\text{a}}$rtner]
Let $\{\nu_{\varepsilon}\}_{\varepsilon>0}$ be a family of probability measures on $\varprojlim X_{\alpha}$ such that $\{\nu_{\varepsilon}\circ p_{\alpha}^{-1}\}_{\varepsilon>0}$ satisfies LDP on every $X_{\alpha}$ with a good rate function $I_{\alpha}$. Then $\{\nu_{\varepsilon}\}_{\varepsilon>0}$ satisfies LDP with a good rate function $I'(x) = \sup\limits_{\alpha}I_{\alpha}(p_{\alpha}(x))$, $x \in\varprojlim X_{\alpha}$.
\end{theorem}
Therefore, the theorem asserts that $\{\nu_{\kappa}\}_{\kappa>0}$ satisfies LDP on the projective limit $\left(\varprojlim S_T, \sigma(\tau_p)\right)$ with a good rate function 
\begin{equation*}
I'(\Gamma) = \sup\limits_{T\in\mathbb{R}_{+}} I_T(\gamma_T) =\sup\limits_{T\in\mathbb{R}_{+}} I_T(\lambda^\gamma)   =  \frac{1}{2}\int\limits_{0}^{\infty}\left(\frac{d\lambda^{\gamma}}{dt}\right)^2 dt \text{ for any } \Gamma \in \varprojlim S_T, 
\end{equation*}
if $\lambda^\gamma$ is absolutely continuous and infinity otherwise.  Next, to translate this result to $\left(S, \sigma(\tau)\right)$ we use contraction principle.

\begin{theorem}[Contraction principle]\label{contraction principle}
If a family of probability measures $\left\{\nu_\varepsilon\right\}_{\varepsilon>0}$ satisfies LDP on a space $\mathcal{X}$ with the good rate function $I':\mathcal{X}\to [0,\infty]$ and the mapping $f:\mathcal{X}\to\mathcal{Y}$ is continuous, then $\left\{\nu_\varepsilon\circ f^{-1}\right\}_{\varepsilon>0}$ satisfies LDP on $\mathcal{Y}$ with the good rate function $I(y) = \inf\left\{I'(x): x\in\mathcal{X}, y = f(x)\right\}$.
\end{theorem}
Applying  this principle to our setting we deduce that $\left\{\mu_\kappa\right\}_{\kappa>0} = \left\{\nu_\kappa\circ p^{-1}\right\}_{\kappa>0}$ satisfies LDP on $(S, \sigma(\tau))$ with a good rate function 
\begin{equation*}
I(\gamma)  =\inf\left\{I'(\Gamma): \Gamma \in \varprojlim S_T, \gamma = g(\Gamma)\right\} = I'(g^{-1}(\gamma))=  \frac{1}{2}\int\limits_{0}^{\infty}\left(\frac{d\lambda^{\gamma}}{dt}\right)^2 dt \text{ for any } \gamma \in S, 
\end{equation*}
if $\lambda^\gamma$ is absolutely continuous and infinity otherwise. This concludes the extension of LDP to the topological space $\left(S, \tau\right)$ and completes the proof of Theorem \ref{LDP_theorem}.

\section{Further comments}\label{sect:further-comments}
As was noted in the introduction the large deviation statement for SLE curves depends on the topology. For example, one could consider the SLE curves as 
		\begin{enumerate}[itemsep=0ex]
		\item Subsets of the upper half-plane $\mathbb{H}$;
		\item Continuous curves in $\mathbb{H}$;
		\begin{enumerate}[topsep=0pt, itemsep=0ex]
				\item modulo reparametrizaton;
		\item in the half-plane capacity parametrization;
		\item in the Natural Parametrization;
		\end{enumerate}
		\item $p$-variation paths;
		\item Elements of a suitable Besov space.
		\end{enumerate}
		
	The first item was dealt with in \cite{PeltolaWang}, where the authors considered the SLE curves as subsets of the upper half-plane and measured distances with the Hausdorff metric. The second item (b) was the objective of the present paper (and (a) follows as well). It would be interesting to try to prove an LDP using the Natural Parametrization, i.e., for SLE$_\kappa$ -- the $d$-dimensional Minkowski content, and for finite Loewner energy curves -- the  arclength.  However, more work would be needed to address this problem, e.g., since  analytic properties of Loewner curves, in particular SLE, depend on the choice of parametrization (cf. \cite{Optimal} and \cite{Lawler2012basic}). 

	It would also be interesting to try to consider the SLE curves as $p$-variation paths. It was established in \cite{Friz_Tran_regularity} that SLE$_\kappa$ (for $\kappa \neq 8$) enjoys $p$-variation regularity
		\begin{equation*}
		||\gamma^\kappa||_{p\text{-var}, [0,1]} = \left(\sup\limits_{\mathcal{P}} \sum\limits_{i=1}^{|\mathcal{P}|}|\gamma(t_i) -\gamma(t_{i-1})|^p\right)^{1/p}<\infty \text{ for all }p> \min\left(1+\frac{\kappa}{8}, 2\right),
		\end{equation*}  
		where $\mathcal{P}$ is a partition of $[0,1]$. Approach based on the $p$-variation distance offers a parametrization-free study of SLE curves since the topology generated by the induced metric includes all parametrization-dependent topologies. We refer to the same paper for a discussion of Besov regularity.

There is also alternative way to prove LDP for SLE in the uniform topology based on \textit{the inverse contraction principle}. One should start with LDP in a weaker topology, for example, the Hausdorff topology; this result was proved in the work of E. Peltola and Y. Wang \cite{PeltolaWang}, and then lift the topology to a stronger one by showing exponential tightness of SLE measures, namely 
\begin{equation*}
\forall M>0 \ \exists \text{ compact set } K_M\subset S_T: \uplim_{\kappa\downarrow 0}\kappa\log\mathbb{P}\left[\gamma^\kappa \in K_M^c\right]  < - M.
\end{equation*} 
It seems that deriving exponential tightness would require the same machinery as in our proof for closed sets, where the main work was to estimate probability of the derivative estimate for the inverse SLE map. However, once we have obtained that probability estimate it is no problem to work out estimates for both open and closed sets. 
 
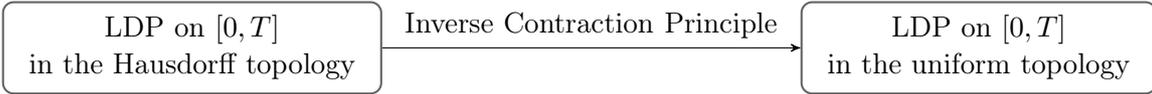
\begin{figure}[H]
\centering
\begin{tikzpicture}
[
node distance=5.5cm,
squarednode/.style={rectangle,rounded corners, draw=black!60, thick, minimum size=10mm}
]
\node[squarednode]  (Hausdorff)  {\begin{tabular}{c} LDP on $[0,T]$\\
in the Hausdorff topology\end{tabular}};
\node[squarednode]   (uniform) [right=of Hausdorff]  {\begin{tabular}{c} LDP on $[0,T]$\\
in the uniform topology\end{tabular}};
\draw[->]  (Hausdorff) to node[above,align=center] { Inverse Contraction Principle} (uniform);

\end{tikzpicture}
\end{figure}

On a different note, the technique used in the proof for open sets can be applied to derive the support theorem for SLE that was shown by H. Tran and Y. Yuan in \cite{SLEsupport}. Unfortunately the bound (\ref{E_complement}) is not well suited for taking the limit $n \to \infty$ and at this stage the proof for SLE support goes through only for small $\kappa$. Nevertheless, one can try to improve the moment estimate (\ref{expectation_bound}) in order to eliminate the factor $4^n$ in (\ref{E_complement}), so that the bound would be well adapted for taking the limit $n\to\infty$.

\section{Appendix: distortion estimate in the rectangle}\label{appendix}

The following lemma originates from \cite{Optimal}.
\begin{lemma}\label{distortion_in_S}
Let $z_{1},z_2 \in S = \{x+iy: x\in \left[-1,1\right], y\in \left[0,1\right] \}$ and assume $\text{Im}z_2, \text{Im}z_1\ge y$. Then for any conformal map $f:2S\to\mathbb{C}$ there are universal constants $c_1, c_2$, both greater than one, such that
\begin{equation*}
|f'\left(z_2\right)| \le c_1 y^{-c_2} |f'\left(z_1\right)|.
\end{equation*}

\end{lemma}	
	
\begin{proof}
Let $\left(S_{j,k}\right)_{j,k}$ be the dyadic decomposition of the domain S, i.e., the collection of rectangles given by
\begin{equation*}
S_{j,k} = \left\{x+iy: x\in \left[\frac{j}{2^{k}}, \frac{j+1}{2^{k}}\right], y\in \left[2^{-\left(k+1\right)}, 2^{-k}\right]\right\}.
\end{equation*}
Pick one rectangle $S_{j,k}$ and consider two points $z,w \in S_{j,k}$ inside. They are at most distance $D = \sqrt{5}\cdot2^{-\left(k+1\right)}$ away from each other. To compare the derivatives at these points we apply Koebe distortion theorem, Lemma~\ref{Koebe}. The point closest to the boundary of $2S$, say $z$, is at least distance $d = 2^{-(k+1)}$ away from it. In order to cover the distance between $z$ and $w$ by intervals of length $rd$, with $r\in(0,1)$, we need at most $\lceil D/rd\rceil = \lceil\sqrt{5}/r \rceil $ of them.  That is we need to apply Koebe distortion $\lceil\sqrt{5}/r \rceil $ times. Hence, for any $ w,z \in S_{j,k}$

\begin{equation}\label{distortion_in_one_square}
|f'\left(w\right)| \le c(r) |f'\left(z\right)|, \text{ with } c(r) = \left(\frac{1+r}{\left(1-r\right)^3}\right)^{\lceil \sqrt{5}/r\rceil}.
\end{equation}
Note that the constant in the inequality does not depend on our choice of $S_{j,k}$, it is uniform for all rectangles. This is the main advantage of the dyadic  decomposition.

Now we return to initial points $z_1 \text{ and } z_2$ we started with. Both of them are at least $y$ away from the real line by the assumption. Let $k\in\mathbb{N}$ be such that $y\in [2^{-\left(k+1\right)}, 2^{-k}]$, that is,
\begin{equation*}
k = \left\lfloor \log_{2}\frac{1}{y}\right\rfloor.
\end{equation*} 
There exists a path from $z_1$ to $z_2$ that goes through at most $2\left(k+1\right)$ squares. Hence, applying the estimate (\ref{distortion_in_one_square}) $2\left(k+1\right)$ times we obtain 
\begin{equation*}
|f'\left(z_2\right)| \le c_1 y^{-c_2} |f'\left(z_1\right)|,
\end{equation*}
where $c_1 =c(r)^2$ and $c_2 = 2\log_2c(r)$. For example, if $r$ is chosen to be $1/2$, then $c_1 = 12^{10}$ and $c_2 = 2\log_2 12$.
\end{proof}

\newpage
\bibliographystyle{plain}
\bibliography{references}
\end{document}